\documentclass[12pt,a4paper]{amsart}
\makeatletter
\renewcommand\normalsize{%
    \@setfontsize\normalsize{11.7}{14pt plus .3pt minus .3pt}%
    \abovedisplayskip 10\p@ \@plus4\p@ \@minus4\p@
    \abovedisplayshortskip 6\p@ \@plus2\p@
    \belowdisplayshortskip 6\p@ \@plus2\p@
    \belowdisplayskip \abovedisplayskip}
\renewcommand\small{%
    \@setfontsize\small{9.5}{12\p@ plus .2\p@ minus .2\p@}%
    \abovedisplayskip 8.5\p@ \@plus4\p@ \@minus1\p@
    \belowdisplayskip \abovedisplayskip
    \abovedisplayshortskip \abovedisplayskip
    \belowdisplayshortskip \abovedisplayskip}
\renewcommand\footnotesize{%
    \@setfontsize\footnotesize{8.5}{9.25\p@ plus .1pt minus .1pt}%%
    \abovedisplayskip 6\p@ \@plus4\p@ \@minus1\p@
    \belowdisplayskip \abovedisplayskip
    \abovedisplayshortskip \abovedisplayskip
    \belowdisplayshortskip \abovedisplayskip}
\ifdefined\pdfpagewidth
\else
\fi
\calclayout
\makeatother

\usepackage[utf8]{inputenc}
\usepackage{amsmath}
\usepackage{enumerate}
\usepackage{amssymb}
\usepackage{amsthm}
\usepackage{bm}
\usepackage{bbm}
\usepackage{graphicx}
\usepackage{color}
\usepackage{accents}
\usepackage{epsfig}
\usepackage{amsmath,amssymb}

\newtheorem{Th}{Theorem}[section]
\newtheorem{lem}[Th]{Lemma}

\theoremstyle{definition}
\newtheorem{Def}[Th]{Definition}

\newtheorem{Cor}[Th]{Corollary}
\newtheorem{Prop}[Th]{Proposition}

\theoremstyle{remark}
\newtheorem*{rem}{\bf Remarks}

\newtheorem*{que}{\bf Questions}
\newtheorem*{conj}{\bf Conjecture}
\newtheorem*{Cond1}{Helson-Szeg\"{o} condition}
\newtheorem*{Cond2}{Uniform Helson-Szeg\"{o} condition}
\newtheorem*{thank}{\ \ \ \bf Acknowledgment}
\numberwithin{equation}{section}

\newcommand{\tend}[3][]{\xrightarrow[#2\to#3]{#1}}

\newcommand{\ds}{\displaystyle}

\newcommand{\R}{\mathbb{R}}

\newcommand{\1}{\mathbbm{1}}

\newcommand{\Z}{\mathbb{Z}}

\newcommand{\N}{\mathbb{N}}
\newcommand{\T}{\mathbb{T}}
\newcommand{\C}{\mathbb{C}}
\newcommand{\D}{\mathbb{D}}
\newcommand{\A}{\mathcal{A}}

\newcommand{\B}{\mathcal{B}}

\newcommand{\Rep}{\textrm{Re}}
\newcommand{\Log}{\textrm{Log}}

\newcommand{\setdef}{\stackrel {\rm {def}}{=}}
\title[Spectral ergodic Banach problem]{Spectral ergodic Banach problem and flat polynomials}
\author[e. H. el Abdalaoui]{el Houcein el Abdalaoui$^{\flat}$}
\address{University of Rouen Normandy
	Department of Mathematics, LMRS  UMR 60 85 CNRS\\
	Avenue de l'Universit\'e, BP.12
	76801 Saint Etienne du Rouvray - France .}
\email{elhoucein.elabdalaoui@univ-rouen.fr; elabdelh@gmail.com}
\urladdr{http://www.univ-rouen.fr/LMRS/Persopage/Elabdalaoui/}

\dedicatory{}

\begin{document}
	
	\begin{abstract} We construct a sequence of flat polynomials with coefficients $0,1$. We thus get that there exist a sequence of
		Newman's polynomials that are $L^\alpha$-flat, $0 \leq \alpha <2$. This settles an old question of Littlewood. In the opposite direction,
		we prove that the Newman polynomials are not $L^\alpha$-flat, for $\alpha \geq 4$.
		We further establish that there is a conservative, ergodic, $\sigma$-finite measure preserving transformation  
		with simple Lebesgue spectrum. This answer affirmatively a long-standing problem of Banach 
		from the Scottish book. Consequently, we obtain a positive answer to  Mahler's problem in the class of Newman polynomials, and this allows us  
		also to answer a question raised by Bourgain on the supremum of the $L^1$-norm of $L^2$-normalized idempotent polynomials.
		%$$
		%{}^{\flat}\includegraphics[scale=0.1]{mafigure.eps}
		%$$
	\end{abstract}
	\maketitle
	\newpage
	\section{Introduction}\label{intro}
	
	The main purpose of this paper is to produce a sequence of flat polynomials with coefficients $0,1$. This answer
	an old question of Littlewood on finding a sequence of analytic trigonometric polynomials
	with coefficients $0,1$ that are $L^1$-flat. Roughly speaking, a sequence of analytic trigonometric polynomials
	is flat if it is converge in some sense to the constant $1$.\\
	
	Here, we construct a sequence of analytic trigonometric polynomials with coefficients $0,1$
	that are $L^\alpha$-flat,
	$0<\alpha<2$. This is done by appealing to the combinatorial Singer's construction combined with Marcinkiewicz-Zygmund
	interpolation inequalities \cite[p.28, chp. X]{Zygmund} and its refinements. Form this, we obtain that there exist
	a sequence of analytic trigonometric polynomials with coefficients $0,1$ that are flat in the almost everywhere sense.\\
	
	Marcinkiewicz-Zygmund interpolation inequalities and its extensions lies in the heart of the interpolation theory.
	It follows that the Hardy spaces and the Carleson interpolation theory play an important role in our proofs.
	For a nice account on the interpolation theory and the $H^p$ theory,
	we refer the reader to \cite[p.147, Chap. 9]{Duren} and \cite[p.275, Chap. 7]{Garnett}, \cite[p.194]{Hoffman}, \cite[p.328]{Rudin}.\\
	
	Our construction benefited also from the ideas of Ben Green and \linebreak Gowers related to the flatness problem in connection with
	Singer and Sidon sets \cite{Atiya}, \cite{BlogG}. We also take advantage from the recent investigations on
	the Marcinkiewicz-Zygmund inequalities and its refinements \cite{ZhongII}, \cite{Zhong}, \cite{Lub}. \\
	
	However, our methods breaks down for the polynomials with coefficients $\pm 1$. Thus, we are not able to answer the weaker
	form of Littlewood question on the existence of $L^1$-flat polynomials with coefficients $\pm 1$.\\
	
	%We further stress that the proof given here is completely different than the proof given in the previous version of this work \cite{A-Arxiv}.
	%Although, it was stated in \cite{A-Arxiv} that one can give an alternative proof of the main results using a Carlson 
	%interpolation theory combined with the methods of disturbed root of unit due to C. Chui and Zhong \cite{ZhongII}. So, here, our main task is to present this alternative proof .\\
	
	The flatness problem was initiated by  Erd\"{o}s \cite{Erdos-L} ,\cite[Problem 22 and 26.]{Erdos-P}) and Littlewood \cite{Littlewood}, and it has a long history. 
	Erd\"{o}s and Newman \cite{Erdos} asked if there is a positive absolute constant $C$ such that
	$$\max_{|z|=1}\Big|\sum_{j=0}^{n}a_j z^j\Big| \geq (1+C) \sqrt{n}, \eqno(EN)$$
	where $|a_j|=1$.\\ 
	
	Subsequently, Littlewood conjectured \cite{Littlewood}, \cite[Problem 9, p.29]{Littlewood-B} that there exist a sequence of the polynomials on the circle $P_n(z)=\sum_{j=0}^{n-1}\epsilon_j z^{n_j}$ with $\epsilon_j =\pm 1$ such that
	\begin{eqnarray}\label{Littlewood1}
	A_1 \sqrt{n} \leq |P_n(z)| \leq  A_2 \sqrt{n},
	\end{eqnarray}
	where $A_1, A_2$ are positive absolute constants and uniformly on $z$ of modulus $1$. 
	Nowadays the analytic polynomials on the circle with  $\pm 1$ coefficients are called Littlewood polynomials \footnote{Five years after this paper was posted on Arxiv, P. Balister and al. established this conjecture \cite{Bal-al}. This is accomplished by exploiting the nice proprieties of the Rudin-Shapiro polynomials combined with Spencer's six deviations lemma. Precisely,  the authors constructed a flat polynomials in the Littlewood sense \cite{Bal-al}. However, it can be seen that those polynomials are not $L^\alpha$-flat, for any $\alpha \geq 0$. This follows directly from Theorem 2.4 and Lemma 3.3 in \cite{Bal-al}.}.\\
	
	Later, Beller and Newman produced a sequence of polynomials $P_n$ with degree $n$ and coefficients
	bounded by $1$ satisfying $\min_{|z=1|}\big|P_n\big| \geq C \sqrt{n}$ \cite{Beller-N1}. In the opposite, 
	J-P. Kahane  disproved the conjecture $(EN)$ \cite{Kahane}.\\
	
	Besides, T. K\"{o}rner \cite{Korner} obtained a positive answer to the Littlewood question \eqref{Littlewood1} 
	in the class of polynomials with coefficients of modulus one by appealing to the Byrnes's construction \cite{Byrnes}.
	But, it turns out that the main ingredient form \cite{Byrnes} used by  K\"{o}rner is not valid \cite{Queffelec}. However,
	Kahane's proof does not used this  flaw argument.\\
	
	Afterwards, J. Beck proved that one can obtain a positive answer to \linebreak Littlewood question \eqref{Littlewood1} by
	producing  a sequence of
	polynomials from the class of polynomials of degree $n$ whose coefficients are
	$400$th roots of unity \cite{Beck}. J. Beck's construction is essentially based on the random construction
	of Kahane.\\
	
	Since then, it was a long standing problem to obtain effective construction of ultraflat polynomials until solved very recently by  Bombieri and Bourgain \cite{Bom}. 
	For a deeper treatment on the Kahane ultraflat polynomials, we refer the reader to
	\cite{Queffelec}. \\
	
	The third extremal problem in the class of analytic trigonometric polynomials concern $L^1$-flatness problem. 
	This problem seems to be mentioned first in \cite{Newman1}. Therein, Newman mentioned the following conjecture:
	\begin{conj}[Newman \cite{Newman1}]
		For any Littlewood polynomial $P$ of degree $n$, $\big\|P(z)\big\|_1<c \sqrt{n+1}$, where $c<1$.\\
	\end{conj}
	Newman in his 1965's paper \cite{Newman2} solved the problem of $L^1$-flatness in the class of analytic trigonometric polynomials with
	coefficients of modulus 1. He proved that the  Gauss-Fresnel polynomials are $L^1$-flat. We refer to \cite{Abd-Nad2} for a simple proof. 
	This result has been strengthened by Beller \cite{Beller}, and Beller \& Newman in \cite{Beller-N2} by proving that the sequence of 
	the Mahler measure of the $L^2$ normalized Gauss-Fresnel polynomials converge to one.\\
	
	For the polynomials with random coefficients $a_k \in \{+1,-1\}$, Salem and Zygmund \cite{Zygmund-Salem}
	proved that for all but $o(2^n)$ choices of $a_k = \pm 1$,
	$$c_1 \sqrt{n \ln(n)}<\bigg\|\sum_{k=0}^{n}a_k z^k\bigg\|_{\infty} <c_2 \sqrt{n \ln(n)},$$
	for some absolute constant $c_1,c_2>0$. Hal\'{a}sz   \cite{Halaz} strengthened this result by
	proving
	$$\bigg\|\sum_{k=0}^{n}a_k z^k\bigg\|_{\infty}=\big(1+o(1)\big) C\sqrt{n \ln(n)},$$
	For some absolute constant $C>0$. Byrnes and Newman computed $L^4$-norm of those polynomials \cite{Byrnes-N}.
	Later, Browein \& Lokhart \cite{Borwein-lo}, and Choi \& Erd\'elyi  \cite{Edyli-Cho} used the central limit theorem to compute the limit of the $L^p$-norm and 
	the Mahler measure of the polynomials with random coefficients $\pm 1$. Their results can be linked to 
	the recent results of Peligrad \& Wu \cite{Peligrad-Wu}, Barrera \& Peligrad, Cohen \& Conze \cite{Cohen-Conze} and Thouvenot \& Weiss \cite{Th-Weiss}. Therein, 
	the authors investigated a dynamical approach with dynamical coefficients, that is,
	$a_k=f(T^kx)$, where $T$ is a measure-preserving transformation on some probability space and $f$ is a square-integrable function.\\
	
	The polynomials with coefficients $a_k \in\{0,1\}$ and the constant term equal to 1 are nowadays called Newman polynomials.
	We further notice  that those polynomials are also known as idempotent polynomials, and since we are concern with $L^1$-flatness, we may assume that the constant term is 1.\\
	
	Let us notice also that the flatness problem is connected to the number theory and to some practical issues arising in the design of a mobile cellular wireless OFDM system \cite{cel}. 
	Consequently, it is related to some engineering issues \cite{dow}, \cite{Borwein1}, \cite{Borwein2}.\\

	It turns out that there is a connection between flat polynomials and the existence of dynamical system with simple Lebesgue spectrum. Bourgain \cite{Bourgain} established that 
	the singularity of the spectrum
	of the so-called rank one maps is related the  supremum of
	the $L^1$-norm of the $L^2$-normalized analytic trigonometric polynomials with coefficients $0,1$.  
	Subsequently, Guenais \cite{Guenais} proved that there exist
	measure-preserving transformations with Lebesgue component of multiplicity one in its
	spectrum provided there exists a sequence of analytic trigonometric polynomials $P_n$ , $n = 1, 2, \cdots$, with coefficients $\pm 1$ whose absolute values $|P_n|$,
	$n = 1, 2,\cdots ,$ converge to 1 in $L^1$.  Here, according to the results from \cite{Abd-Nad}, our polynomials provide us $L^1$-flat polynomials which can be used to construct 
	a conservative, ergodic, $\sigma$-finite measure preserving transformation on a Lebesgue space
	with simple Lebesgue spectrum. This gives an affirmative answer to a long-standing spectral problem of Banach
	from the Scottish book. Our approach  allows us also to obtain an affirmative answer to the
	problem of Mahler (whether there exists a sequence of analytic trigonometric polynomials with integer coefficients for which
	the Mahler measure of the $L^2$-normalized sequence converges to 1 \cite[p.6]{Borwein-b}, \cite{Borwein1}, \cite[Theorem 1]{mahler}(therein, see the footnote!). We further get that the  supremum of
	the $L^1$-norm of the $L^2$-normalized analytic trigonometric polynomials with coefficients $0,1$ is $1$. This answer a question raised by J. Bourgain \cite{Bourgain}. \\

	We stress that this paper is deeply indebted to the investigation started in \cite{Abd-Nad}, \cite{Abd-Nad1} and \cite{Abd-Nad2}.
	So, it is may seen as a companion to those papers.\\

	For the convenience of the reader, we repeat the relevant material from \cite{Abd-Nad},\cite{Abd-Nad2} and
	\cite{Zygmund}, without proofs, thus making our exposition self-contained.\\
	
	The paper is organized as follows. In section 2, we state our main results. In section \ref{flat-P}, we present several definitions of flatness in the class of analytic
	trigonometric polynomials and the fundamental characterization of $L^1$-flatness. 
	In section \ref{Sidon}, we recall the notion of Singer and Sidon sets in number theory, 
	and we establish that the $L^2$-normalized Newman polynomials are not $L^\alpha$-flat, for $\alpha \geq 4$.
	In section \ref{main1-P}, we give the proof of our first main result.
	In section \ref{Rieszp}, \ref{ergodic}, we recall the definition of Riesz products and the connection to the ergodic theory. 
	Finally, in section \ref{main2-P} we prove our second main result. 
	
	\section{Main results}
	Consider the torus $\T=\Big\{z \in \C~~~:~~|z|=1\Big\}$ equipped with the normalized Lebesgue measure $dz$. Let $n_0<n_1<n_2<\cdots$ be a positive sequence of integers and put
	$$P_n(z)=\sum_{j=0}^{n-1}\epsilon_j \sqrt{p_j}z^{n_j},$$
	with $|\epsilon_i|= 1$ and $(p_0,\cdots,p_{n-1})$ is a probability vector. Such polynomials are raised in the study of the spectral type of some class of dynamical systems in ergodic theory. For more details we refer to \cite{Abd-Nad1} .\\
	
	Here, we restrict ourself to the case $\epsilon_i=1$ and $p_i=\frac1{n}$. We thus concentrated our investigations on the flatness
	%we concentrated our concern on the flatness issue
	problem in the class of polynomials of the from
	$$P_n(z)=\frac1{\sqrt{n}}\sum_{j=0}^{n-1}z^{n_j}.$$
	Following \cite{Abd-Nad}, this class is called class B.\\

	We state our main results as follows.
	
	\begin{Th}\label{main1} There exist a sequence of analytic trigonometric polynomials $\big(P_n\big)_{n \in \N}$ with coefficients $0$ and $1$ such that the polynomials $\frac{P_n(z)}{\|P_n\|_2}$  
		are flat in almost everywhere sense, that is,
		$$\frac{P_n(z)}{\|P_n\|_2} \tend{n}{+\infty}1,$$
		for almost all $z$ in the torus with respect to the Lebesgue measure $dz$.
	\end{Th}
	As a consequence, we obtain the following theorem.
	\begin{Th}\label{main2}There exist a conservative ergodic measure preserving transformation on $\sigma$-finite space $(X,\A,\mu)$
		with simple Lebesgue spectrum.
	\end{Th}
	We further establish the following.
	\begin{Th}\label{main3}
		Any sequence of analytic trigonometric polynomials $\big(P_n\big)_{n \in \N}$ with coefficients $0$ and $1$ is not $L^\alpha$-flat, for any
		$\alpha \geq 4$.
	\end{Th}
	
	We recall that $T$ is a measure preserving transformation on  $\sigma$-finite space if
	$(X,\A,\mu)$ is a Lebesgue space  with $\mu$ is $\sigma$-finite measure and nonatomic, and for any Borel set $A$, we have
	$\mu(T^{-1}A)=\mu(A).$ $T$ is ergodic if every invariant Borel set $A$ under $T$ (i.e. $\mu(T^{-1}A \Delta A)=0.$), we have
	$\mu(A)=0$ or $\mu(A^c)=0$, where $A^c$ is the complement of $A$.\\
	
	The key point in the proof of Theorem \ref{main2} is the construction of a rank-one infinite measure preserving transformation
	with desired properties. This is done by applying the cutting and staking method. We will assume that the reader is familiar with this method and we refer to \cite{Friedman1} for
	a nice account.\\
	
	Theorem \ref{main2} answer affirmatively the long-standing problem attributed to Banach on the existence
	of dynamical system which simple Lebesgue spectrum and with no-atomic measure. This problem is stated in Ulam's book \cite[p.76]{Ulam}
	as follows.
	\begin{que}[Banach Problem]
		Does there exist a square integrable function $f(x)$ and a measure preserving transformation $T(x)$,
		$-\infty<x<\infty$, such that the sequence of functions $\{f(T^n(x)); n=1,2,3,\cdots\}$ forms a complete
		orthogonal set in Hilbert space?\footnote{Professor M. Nadkarni pointed to me that the question contain an oversight. The  sequence of functions should be bilateral, that is, $n \in \Z$.}
	\end{que}
	
	The most famous Banach problem in ergodic theory asks if there is a measure preserving transformation on a probability space which
	has simple Lebesgue spectrum. A similar problem is mentioned by Rokhlin \cite[p.219]{Rokh}. Precisely, Rokhlin asked on the existence
	of an ergodic measure preserving transformation on a finite measure space whose spectrum is Lebesgue type with finite multiplicity.
	Later,  Kirillov in his 1966's paper \cite{Kiri} wrote ``there are grounds for thinking that such examples do not exist".
	However he has described a measure preserving  action (due to M. Novodvorskii) of
	the group $(\bigoplus_{j=1}^\infty{\mathbb {Z}})\times\{-1,1\}$ 
	on the compact dual of discrete rationals whose unitary group has Haar spectrum of multiplicity 2. Similar group actions with higher finite even multiplicities are also given.\\
	
	Subsequently, finite measure preserving transformation having \linebreak Lebesgue component of finite even multiplicity have been
	constructed by J. Mathew and M. G. Nadkarni \cite{MN}, Kamae \cite{Kamae}, M. Queffelec \cite{Q}, and  O. Ageev \cite{Ag}.
	Fifteen years later, M. Guenais \cite{Guenais} used a $L^1$-flat generalized Fekete polynomials on some torsion groups 
	to construct a group action with simple Lebesgue component. 
	A straightforward application of Gauss formula yields that the generalized Fekete polynomials constructed by Guenais
	are ultraflat. Very recently, el Abdalaoui and Nadkarni strengthened Guenais's result \cite{Abd-Nad2} by proving that there exist an ergodic non-singular dynamical 
	system with simple Lebesgue component. However, despite all these efforts, it is seems that the question of Rokhlin still open since the maps constructed does not have {\it a pure
		Lebesgue spectrum.}
	\\

	Our methods is far from making any contribution to this problem. At now, it is seems that this problem is a ``dark continent" for
	the ergodic theory.\\

	Nevertheless, our results allows us to answer a question raised by Bourgain on the  supremum of the $L^1$-norm over all
	polynomials from class B \cite{Bourgain}. Indeed, Theorem \ref{main1} gives the following
	\begin{Cor} $\displaystyle \beta=\sup_{n>1}\sup_{k_1<k_2<k_3<\cdots<k_n}\Big\|\frac1{\sqrt{n}}\sum_{j=1}^{n}z^{k_j}\Big\|_1=1.$
	\end{Cor}
	In \cite{Abd-Nad}, \cite{Aistleitner} and \cite{Erdyli}, the authors established already that $\beta \geq \frac{\sqrt{\pi}}2$, and, 
	it is easy to see that the simple case $n=2,k_1=0,k_2=1$, gives $\beta \geq \frac{2\sqrt{2}}{\pi}.$ We further have
	\begin{Cor}There exist a sequence of analytic trigonometric polynomials $\big(P_k\big)_{k \in \N}$ with coefficients $0$ and $1$ 
		such the Mahler measure of the polynomials $\frac{P_k}{\|P_k\|_2}$ converge to 1.
	\end{Cor}
	The Mahler measure of analytic trigonometric polynomials $P_k$ is given by
	\[
	M(P_k)=\exp \Big(\int_{\T} \log\big(\big|P_k(z)\big|\big) dz \Big).
	\]
	Using Jensen's formula \cite{Rudin}, it can be shown that
	\[
	M(P_k)=\frac1{\sqrt{m_k}}\prod_{|\alpha|>1} |\alpha|,
	\]
	where, $\alpha$ denoted the zero of the polynomial $\sqrt{m_k}P_k$. In this definition, an empty product is assumed to be $1$ so the Mahler 
	measure of the non-zero constant polynomial $P(x)=a$ is $|a|$. A nice account on the subject may be founded in \cite[pp.2-11]{Ward}, \cite{Borwein-b}.\\
	
	\noindent{}The next proposition list some elementary properties of the Mahler measure, and its proof is left to the reader.
	\begin{Prop}\label{basic}Let $(X,\mathcal{B},\rho)$ be a probability space. Then,
		for any two positive functions $f,g \in L^1(X,\rho)$, we have
		\begin{enumerate}[(i)]
			\item $M_{\rho}(f)\setdef \exp\Big(\rho(\log(f))\Big)$ is a limit of the norms $||f||_{\delta}$ as $\delta$ goes to $0$, that is,  \[
			||f||_{\delta} \setdef \Biggl(\int f^{\delta} d\rho\Biggr)^{\frac1{\delta}} \tend{\delta}{0} M_{\rho}(f),
			\]
			provided that $\log(f)$ is integrable.
			\item If $\rho\Big\{ f >0 \Big\} <1$ then $M_{\rho}(f)=0$.
			\item If $0<p<q<1$, then $\bigl\|f\bigr\|_p \leq \bigl\|f\bigr\|_q$.
			\item If $0<p< 1$, then $M_{\rho}(f) \leq \bigl\|f\bigr\|_p$.
			\item $\ds \lim_{\delta \longrightarrow 0}\int f^\delta d\rho = \rho\Big\{f>0\Big\}.$
			\item $M_{\rho}(f) \leq \bigl\|f\bigr\|_1$.
			\item $M_{\rho}(fg)=M_{\rho}(f)M_{\rho}(g)$.
		\end{enumerate}
	\end{Prop}

	Using the classical Beurling's outer and inner decomposition, \linebreak el Abdalaoui and Nadkarni \cite{Abd-Nad} computed the Mahler measure of $P_k$. 
	They further apply the $H^p$ theory to establish a formula for the Mahler measure of the generalized Riesz product.  We recall a part of this in section \ref{Rieszp}.
	
	\section{flats polynomials}\label{flat-P}
	A sequence $P_j, j =1,2,\cdots$ of trigonometric polynomials is said to be $L^p$-flat if the sequence $\frac{|P_j|}{\|P_j\|_2},$  $j=1,2,\cdots$
	converge to the constant function 1 in the $L^p$-norm, $p \in [1,+\infty],~~p\neq 2$ . If $p=+\infty$ the sequence $P_j$ is said to be ultraflat.\\
	
	The flatness issue can be considered for three class of analytic trigonometric polynomials. The polynomials with non-negative coefficients, 
	the Littlewood polynomials which correspond to the polynomials with coefficients $\pm 1$, and the Newman polynomials which correspond to the polynomials 
	with coefficients $0$ or $1$ and the constant term 1. For all those polynomials, it seems that the existence of $L^p$-flat polynomials is unknown.\\
	
	The following notion of almost everywhere flatness is introduced in \cite{Abd-Nad}.
	
	\begin{Def}A sequence $P_j, j =1,2,\cdots$ of trigonometric polynomials with $L^2$-norm one is said to be flat almost everywhere, if $P_j(z)$ converge almost everywhere to 1 with respect to $dz$.
	\end{Def}
	Applying Vitali convergence theorem \cite{Rudin} one can see that if $P_j$ is almost everywhere flat then $P_j$ is $L^1$-flat. In the opposite direction, 
	if $P_j$ is $L^1$-flat then one can drop a subsequence over which $P_j$ is almost everywhere flat.\\
	
	We further have the following.
	\begin{Prop}\label{mainP}Let $(P_n)_{n \in \N}$ be a sequence of analytic trigonometric polynomials with $L^2$-norm one. Then, the following are equivalent
		\begin{enumerate}
			\item $\big(P_n\big)$ is $L^1$-flat,
			\item $\big(\|P_n\|_1\big)$ converge to 1.
		\end{enumerate}
		Moreover, if $\big(P_n\big)$  is almost everywhere flat then
		$$\||P_n|^2-1\|_1 \tend{n}{+\infty}0.$$
	\end{Prop}
	\begin{proof} (1) implies (2) is straightforward. For (2) implies (1), notice that
		\begin{eqnarray}\label{eqffund}
		\||P_n|-1\|_2^2=1=2\big(1-\|P_n\|_1\big).
		\end{eqnarray}
		For the last fact, by Cauchy-Schwarz inequality, we have
		$$
		\||P_n|^2-1\|_1 \leq \||P_n|-1\|_2 \||P_n|+1\|_2 \leq 2 \||P_n|-1\|_2.$$
		The last inequality is due to $\|P_n\|_2=1$ combined with the triangle inequality. Thus, it is suffice to see that
		$$\||P_n|-1\|_2 \tend{n}{+\infty} 0.$$
		But, by \eqref{eqffund}, this is equivalent to $(\|P_n\|_1)_{n \in \N}$ converge to 1 which
		follows from the Vitali convergence theorem.
	\end{proof}

	At this point, let us notice that the classical strategy introduced by Newman and Beller to produce the $L^1$-flat polynomials 
	is based on the reduction of the problem of $L^1$-flatness to $L^4$-flatness problem. For the polynomials form class $B$ this strategy fails. This is proved in \cite{Abd-Nad}. In the next section we will prove much more by establishing Theorem \ref{main3}.
	
	\section{Sidon sets, Singer sets and flatness}\label{Sidon}
	Let $R$ be a positive integer and let
	$S=\{s_1<s_2<s_3<\cdots<s_R\}$ be a subset of $[0,R)$. Following Erd\"{o}s \cite{Erdos},
	\begin{Def} A set $S$ is called a Sidon set if all the sums $s+t$, $s \leq t \in S,$ are distinct.
	\end{Def} 

	Sidon asked on the maximal size of the Sidon set subset of $\{1,\cdots,R\}$. Erd\"{o}s and Tur\'{a}n \cite{Erdos} proved that if $S \subset [0,R]$ is a Sidon set then
	$$|S|<\sqrt{R}+10\sqrt[4]{R}+1.$$
	Lindstr\"{o}m  strengthened this result and proved \cite{Lindstrom}
	$$|S|<\sqrt{R}+\sqrt[4]{R}+1.$$
	In the other direction it has been shown by Chowla \cite{Chowla} and Erd\"{o}s using a theorem of Singer \cite{Singer} that
	$$|S| \geq \sqrt{R}-o(\sqrt{R}).$$
	
	Nowadays, it is customary to use algorithmically Singer's theorem to produce a Sidon subset of the given set $\{1,\cdots,N\}$.
	Furthermore, the construction can be used to produce a Sidon subset with some desired proprieties of its sumsets \cite[p.83]{roth},
	\cite{Rusza1}. We notice that Singer established his theorem in the finite projective geometry setting.
	In the number theoretic setting, the theorem can be stated as follows.
	
	\begin{Th}[Singer \cite{Singer}] Let $p$ be a prime and let $q=p^2+p+1$. Then, there exist $A \subset \Z/q\Z$ with
		$|A|=p+1$ such that for all $x \in \Z/q\Z\setminus\{0\},$ there exist $a_1,a_2$ such that $x=a_1-a_2$.
	\end{Th}
	Such set, in which every non-zero difference $\textrm{mod}~~q$ arises exactly one is called a perfect difference set or Singer set.
	For the construction of Singer set, we refer the reader to \cite{Singer}. Using the Singer sets, Erd\"{o}s-S\'{a}rk\"{o}zy-S\'{o}s  \cite{Erdos-S}, \cite{Erdos-S2} and Rusza
	\cite{Rusza1}, \cite{Rusza2} constructed a Sidon set $S$ subset of $\{1,\cdots,N\}$ such that
	$$|S| \geq \sqrt{N}-o(\sqrt{N})~~~~~~~~~~~~~~~~~~~~~~~~~~\eqno (ER),$$
	With some desired properties.\\
	
	We notice that Singer's construction is based on the nice properties of finite fields \cite{Singer}.\\
	
	%Let us further mention that the lower bound and the upper bound of the quantities $M(R)$ and $N(R)$ can be obtained by considering the following toy examples.\\
	
	%For the first example we take $s_i=i$. This gives
	%  $$P_S(z)=\frac{1}{\sqrt{R}}\sum_{i=1}^{R}z^i,$$
	%and
	%  $$\Big|P_S(z)\Big|^2=1+\frac1{R}\sum_{l =1}^{R}(R-l)z^l+\frac1{R}\sum_{l=1}^{R}(R+l)z^{-l}.$$
	%Therefore $M(R)=R-1$ and $N(R)=R$. We thus have $M(R)$ is maximal and $N(R)$ is minimal. Indeed,
	%For any $ n \in \N^*$, the number of solution of the equation $n=s_j-s_i$ is less than $R-n$ since
	%any solution $(s_j,s_i)$, when it exists, satisfy $ n \leq s_j \leq R$.\\
	
	%The second example correspond to the case for which the support of the Fourier transform is a Sidon subset $S$ of $[1,R]$. In this case
	%$M(R)=1$ and $N(R)=\frac{|S|(|S|-1)}{2}$. Indeed, by definition of the Sidon set all $(s_j-s_i)$ are distinct. Hence,
	%the first row of the matrix $\M_S$ contain $R-1$ elements, the second row $R-2$, and the last row one element. By adding, we get
	%$$(R-1)+(R-2)+\cdots+1=\frac{R(R-1)}{2}.$$
	%Whence $M(R)$ is minimal and $N(R)$ is maximal. It is seems that the quantity $M(R)N(R)$ is balanced.\\
	
	%\paragraph{\textbf{On $L^4$-norm strategy and Newman polynomials.}}
	%It is hidden in the proof given by  Chidambaraswamy \cite{Chidam} that the $L^2$-norm of the polynomials
	Let us now present the proof of Theorem \ref{main3}. We start by claiming that
	$(\big\||P_n(z)|^2-1\big\|_2)$ does not converge to $0$. Indeed, write
	$$\Big|P_n(z)\Big|^2-1=\frac1{n}\sum_{l=1}^{N(n)}m(l)z^{r_l}+ \frac1{n}\sum_{l=1}^{N(n)}m(l)z^{-r_l},$$
	where $m(l)$ is the multiplicity of $r_l$ given by
	$$m(l)=\Big|\Big\{(i,j)~~~:s_j-s_i=r_l\Big\}\Big|,$$
	and $r_l$ is defined by
	$$\Big\{s_j-s_i,~~j<i\Big\}=\Big\{r_1,r_2,\cdots,r_{N(n)}\Big\}.$$
	Whence
	\begin{eqnarray*}
		\Big\||P_n(z)|^2-1\Big\|_2^2
		&=&\frac2{n^2}\sum_{j=1}^{N(n)}m(j)^2\\
		&\geq&  \frac2{n^2} \sum_{j=1}^{N(n)}m(j)\\
		&\geq& \frac2{n^2} \frac{n(n-1)}2,
	\end{eqnarray*}
	since
	$$\sum_{j=1}^{N(n)}m(j)=\frac{n(n-1)}2.$$
	Therefore
	$$\Big\||P_n(z)|^2-1\Big\|_2^2 \geq 1+\frac1{n},$$
	
\noindent	which complete the proof of the claim. From this, it easy to see that $\|P_n\|_4 \geq 2$, for any $n$. Hence,
	$\liminf\|P_n\|_4 \geq 2$. We thus conclude that for any $\alpha \geq 4$, $(P_n)$ are not $L^\alpha$-flat, since otherwise
	$\|P_n\|_\alpha$ will converge to $1$ which is impossible. This achieve the proof of Theorem \ref{main3}.
	\section{Proofs of the first main result (Theorems \ref{main1})}\label{main1-P}
	Let us first outline the main ideas of our strategy.  
	The proof is divided into two part. In the first part, we construct a sequence of analytic polynomials $(P_q)$, $q=p^2+p+1$ and $p$ prime 
	with coefficients given by the indicator function of Singer sets, and we prove that the sequences $|P\big(z_{q,j}\big)|$ where $z_{q,j}$ are not trivial $q$-root of unity, converge to $1$. We 
	proceed next to establish that the sequence of analytic polynomials $(P_q)$ is $L^\alpha$-flat, $\alpha \in (0,2)$. For that, we need to 
	estimate the $L^\alpha$-norms of the sequence of trigonometric polynomials $(|P_q|^2-1)$. We thus use Marcinkiewicz-Zygmund inequalities with the nodal points 
	$\Big\{z_{q,j,\delta}\Big\}_{j=0,\cdots,2q-1}$ given by 
	the $q$-root of unity and its $\delta$-perturbation. This sequence of the nodal points  has the property that the map $\ell \mapsto z_{q,j,\delta}^\ell$, is
	almost $q$-periodic up to $\delta$ (which is supposed to be small), for any $j=0,\cdots,2q-1$. In fact, we established that any $\delta$-perturbation of a given Marcinkiewicz-Zygmund family with $\delta$ small is a Marcinkiewicz-Zygmund family. This is later result is based on the notion of Carleson measure and uniformly separated family. Combining this with the nice properties of Singer sets, we conclude that the sequence $\big(\big\||P_q|^2-1\big\|_\alpha\big)$, $\alpha \in (0,2)$ converge to zero.\\
	
	In the second part, we use the material from \cite{Abd-Nad},\cite{Abd-Nad1} and \cite{Abd-Nad2} to obtain that there exist a conservative map on $\sigma$-finite measure space with simple Lebesgue 
	spectrum. This answer affirmatively Banach question from the Scottish book. As a consequence, we obtain also an answer to Bourgain's question and to Mahler's question.\\
	
	Let us now proceed to the proof.\\

	We start by putting, for any finite subset of integer $A$, 
	$$P_A(z)=\frac1{\sqrt{|A|}}\sum_{a\in A}z^a,~~~~z \in \T,$$
	Where $|A|$ is the number of elements in  $A$. The $L^2$-norm of $P_A$ is one since
	\begin{eqnarray}\label{eqf1}
	\big|P_A(z)\big|^2=1+\frac1{|A|}\sum_{\overset{d=b-a \in A-A}{a \neq b}}z^{d},
	\end{eqnarray}
	where $A-A$ is the set of difference of $A$.\\
	If $|A-A|=\frac{|A|(|A|-1)}{2}$ then $A$ is a Sidon set.
	The nice properties of Singer's set allows us to prove the following.
	\begin{lem}\label{Lem1} Let $p$ be a prime number and $S$ a Singer set of $\Z/q\Z$ with $q=p^2+p+1.$.
		Then for any $r \in \Z/q\Z\setminus\{0\}$, we have
		$$\Big|P_S\big(e^{2\pi i\frac{r}{q}}\big)\Big|=\sqrt{\frac{p}{p+1}}.$$
	\end{lem}
	\begin{proof}Applying \eqref{eqf1} we get
		\begin{eqnarray*}
			\Big|P_S\big(e^{2\pi i\frac{r}{q}}\big)\Big|^2&=&1+\frac1{|S|}\sum_{t=1}^{q-1}e^{2\pi i\frac{t.r}{q}}\\
			&=&1-\frac1{|S|},
		\end{eqnarray*}
		since
		$$\sum_{t=0}^{q-1}e^{2\pi i\frac{t.r}{q}}=1+\sum_{t=1}^{q-1}e^{2\pi i\frac{t.r}{q}}=0.$$
		Therefore, we can write
		$$\Big|P_S\big(e^{2\pi i\frac{r}{q}}\big)\Big|^2=\frac{|S|-1}{|S|}=\frac{p}{p+1},$$
		and the proof of the lemma is complete.
	\end{proof}
	The second main ingredient of our proof is based on the classical \linebreak  Marcinkiewicz-Zygmund inequalities (see
	\cite[Theorem 7.5, Chapter X, p.28]{Zygmund}) and some ideas linked to its recent refinement obtained by Chui-Shen-Zhong \cite{ZhongII} 
	and many others authors. Therefore, we will need  some classical results form the $H^p$ theory and interpolation theory of Carleson.\\
	%by appealing to some classical results form the $H^p$ theory and interpolation theory of Carleson, 
	%we will gives an alternative proof to the proof given in \cite{A-Arxiv}.\\
	
	Let  $D_n, K_n$ and $V_n$ be respectively the Dirichlet kernel,  the Fej\'er kernel and the de La Vall\'e Poussin kernel. We recall that
	$$D_n(x)=\sum_{j=-n}^{n}e^{2\pi i j x}=\frac{\sin\big( \pi (2n+1) x\big)}{\sin(\pi x)},$$
	$$K_n(x)=\sum_{j=-n}^{n}\Big(1-\frac{|j|}{n+1}e^{2\pi i j x}\Big)=\frac1{n+1}\Bigg\{ \frac{\sin\big( \pi (n+1) x\big)}
	{\sin(\pi x)} \Bigg\}^2,$$
	and
	$$V_n(x)=2K_{2n+1}(x)-K_n(x).$$
	We also put
	$$D_n^{*}(e^{2\pi ix})=\sum_{j=0}^{n-1}e^{2\pi i j x}.$$
	We recall that the Poisson kernel $P_r$, $0<r<1$, is given by
	\begin{eqnarray}\label{Poisson}
	P_r(\theta)&=&\sum_{-\infty}^{+\infty}r^{|n|}e^{i n \theta} \nonumber\\
	&=&\frac{1-r^2}{|1-re^{i\theta}|^2}  \\
	&=& \frac{1-r^2}{1-2r\cos(\theta)+r^2}.\nonumber
	\end{eqnarray}
	This kernel is related to the Cauchy kernel $\ds C_r(\theta)\setdef\frac1{1-r e^{i \theta}}$ by the following relation
	$$P_r=\Rep(H_r),{\textrm{~~~where~~}} H_r=2C_r-1.
	$$
	The imaginary part of $H_r$ is called the conjugate Poisson kernel and denoted by
	$$Q_r(\theta)=\frac{2 r \sin(\theta)}{1-2r\cos(\theta)+r^2}.$$
	Let us recall also that if $f=u+i\widetilde{u}$ is analytic in the closed disc with $f(0)$ is real then
	$$f(re^{i\theta})=u*H_r(\theta),$$
	and
	$$ \widetilde{u}(\theta)=u*Q_r(\theta).$$
	We notice that $\widetilde{u}$ is the harmonic conjugate to $u$, which vanishes at the origin, and of course, $Q_r$ is the harmonic conjugate to $P_r$. For $f \in L^1(\T)$, 
	the harmonic conjugate of $f$ is given by
	$$\widetilde{f}(re^{i\theta})=Q_r*f(\theta)=-i\sum_{\overset{n=-\infty}{n \neq 0}}^{+\infty}\frac{n}{|n|}r^{|n|}\widehat{f}(n) e^{int}.$$
	It is well known that the radial limit of $\widetilde{f}(re^{i\theta})$ exist almost everywhere, and this radial limit
	denoted by $\widetilde{f}$ is the conjugate function of $f$.\\
	
	We will use often the following classical property: If $F=\exp(H)$, where $H$ is an analytic function. Then
	$$|F|=\exp(\Rep(H)).$$
	
	Given a continuous function $f$ on the torus $\T$ and a triangular family of equidistant points $\theta_{n,j} \in [0,1), j=0,\cdots 2n$, $n \in \N^*$, that is,
	$$\theta_{n,j}=\theta_{n,0}+\frac{j}{2n+1},~~~~~~~~j=0,\cdots,2n.$$
	We define the Lagrange polynomial interpolation of $f$ at $\{\theta_n,j\}$ by
	$$L_n(f,\{\theta_n,j\})(e^{2 \pi i x})=\frac{1}{2\pi}\int_{0}^{2\pi}f(t)D_n(x-\theta_{n,j}) d\omega_{2n+1}(t),$$
	where $\omega_{2n+1}$ is a function defined by
	$$\omega_{2n+1}(t)=\frac{2\pi j}{2n+1} ~~~~\textrm{~~for~~} \frac{2\pi j}{2n+1}\leq  t < \frac{2\pi (j+1)}{2n+1},
	j=0,\pm 1, \pm 2,\cdots.$$
	$\omega_{2n+1}$ is a step function with jump $\frac{2 \pi}{2n+1}$ at the points $\frac{2\pi j}{2n+1}$ and
	$d\omega_{2n+1}$ its Riemann-Stieltjes integral. In the same manner, we define the step function $\omega_m$, for any $m \in \N^*$ and we denote its Riemann-Stieltjes integral by $d\omega_m$.\\
	
	We will need the following classical inequality due to S. Bernstein and A. Zygmund. For its proof, we refer to
	\cite[Theorem 3.13, Chapter X, p. 11]{Zygmund}.
	\begin{Th}{[Bernstein-Zygmund inequality].}\label{Bernstein} For any $p \geq 1$, for any polynomial $P$ of degree $n$, we have
		
		$$\big\|P'\big\|_{p} \leq n \big\|P\big\|_{p},$$
		
		where $P'$ is the derivative of $P$. The equality holds if and only if $P(e^{ix})=M \cos(nx+\xi).$
	\end{Th}
	M\'{a}t\'e, Nevai and Arestov extended Bernstein-Zygmund inequality by proving that it is valid for $p \geq 0$. \cite[p.142]{Borwein-b}. 
	For $p=0$, the result is due to Mahler, a simple proof can be found in \cite{Ward}. Although we will not need this result in such generality.\\
	
	The Marcinkiewicz-Zygmund  interpolation inequalities assert that for
	$\alpha > 1$, $n \geq 1$, and a polynomial $P$ of degree $\leq n-1$,
	\begin{eqnarray}\label{MZ}
	\frac{A_{\alpha}}{n}\sum_{j=0}^{n-1}\big|P(e^{2\pi i\frac{j}{n}})\big|^{\alpha}
	\leq \int_{\T}\Big|P(z)\Big|^{\alpha} dz \leq \frac{B_{\alpha}}{n}\sum_{j=0}^{n-1}\big|P(e^{2\pi i\frac{j}{n}})\big|^{\alpha},
	\end{eqnarray}
	where  $A_{\alpha}$ and $B_{\alpha}$ are independent of $n$ and $P$.\\

	The left hand inequality in \eqref{MZ} is valid for any non-negative non-decreasing convex function and in the more general form \cite[Remark, Chapter X, p. 30]{Zygmund}. 
	
	%For sake of completeness, we will state and present a sketch of the proof of it.
	
	%\begin{Th}\label{Key}Let $\kappa>0$, $m \geq (1+\kappa) 2n$. Then, for any non-negative, non-decreasing and convex function $\phi$, for any trigonometric polynomial $Q$ of degree $n$, we have
	%$$\int_{0}^{2\pi}\phi(A_{\kappa}|Q|) d\omega_m \leq \int_{0}^{2\pi}\phi(|Q|) dx,$$
	%where
	%$$A_{\kappa}=\frac{1}{1+\kappa^{-1}}.$$
	%\end{Th}
	%\begin{proof}
	%The proof is the same as in \cite[p.29]{Zygmund}.  One only needs to substitute the de Vall\'e-Poussin kernel $V_{n-1}$ by
	%$$V_{n,h}=\Big(1+\frac{n}{h}\Big)K_{n+h-1}-\frac{n}{h}K_{n-1},$$
	%where $h=[2\kappa n]+1$ and $\frac1{3}$ by $\frac1{1+\frac{2n}{h}}$.
	%$V_{n,h}$ is the de Vall\'e-Poussin kernel of order $h$.
	%\end{proof}
	
	%\begin{rem1} In \cite[Theorem 7.28, Chapter X, p. 33]{Zygmund}, one may found  the proof of the right hand inequality in 
	%the Marcinkiewicz-Zygmund inequalities under the same assumptions as in Theorem \ref{Key}.
	%\end{rem1}
	
	The next lemma is crucial for the proof of our main result.
	
	\begin{lem}\label{Lem2}Let $p$ be a prime number and $S$ a Singer set of $\Z/q\Z$ with $q=p^2+p+1.$
		Then, for any $\alpha>1$, we have
		$$\frac1{q}\sum_{r=0}^{q-1}\Big|P_S\big(e^{2\pi i\frac{r}{q}}\big)\Big|^{\alpha}
		=\frac1{q}\Big({(p+1)}^{\frac{\alpha}2}+(q-1)\Big({\frac{p}{p+1}}\Big)^{\frac{\alpha}2}\Big).$$
	\end{lem}
	\begin{proof}It is straightforward from Lemma \ref{Lem1}.
	\end{proof}
	Lemma \ref{Lem2} yields for any $0<\alpha<4$,
	\begin{eqnarray}\label{eqf3}
	\lim_{q \longrightarrow +\infty}\frac1{q}\sum_{r=0}^{q-1}\Big|P_S\big(e^{2\pi i\frac{r}{q}}\big)\Big|^{\alpha}=1.
	\end{eqnarray}
	Furthermore, we have
	\begin{Th}\label{Norm-four} For any $\alpha$ in the interval $ (1,4)$,
		$$\Big\|P_S\Big\|_{\alpha} \leq C_{\alpha}.$$
	\end{Th}
	\begin{proof}Applying Marcinkiewicz-Zygmund  interpolation inequalities \eqref{MZ}, we get 
		\begin{align}
		\Big\|P_S\Big\|_{\alpha} \leq C_{\alpha} \sum_{j=0}^{n-1}\big|P(e^{2\pi i\frac{j}{n}})\big|^{\alpha}.
		\end{align}
		But 
		$$\big|P(1)\big|^{\alpha} \leq |S|^{\frac{\alpha}{2}} \leq q,$$
		and by Lemma \ref{Lem2}, we have
		$$\big|P(e^{2\pi i\frac{j}{n}})\big|<1.$$
		Therefore 
		$$\Big\|P_S\Big\|_{\alpha} \leq C_{\alpha}.$$
		The proof of the theorem is complete.
	\end{proof}	
	Now, following the strategy in \cite{ZhongII}, we perturb the root of unity as follows. Put
	\begin{eqnarray*}
		t_{q,r}&=&\frac{r}{q}, \textrm{~~and~}\\
		t_{q,r}^*&=&\frac{r}{q}\pm \frac{\delta}{q.p^{1/2+\epsilon}}, ~~~~\delta>0, \epsilon>0.
	\end{eqnarray*}
	We thus have the following.
	\begin{lem}\label{Lem3}For any $0<\alpha< 4$,
		$$\lim_{q \longrightarrow +\infty}\frac1{q}\sum_{r=0}^{q-1}\Big|P_S\big(e^{2\pi i t_{q,r}^*}\big)\Big|^{\alpha}=1.$$
	\end{lem}
	\begin{proof}Applying Bernstein theorem (Theorem \ref{Bernstein}), we get
		\begin{eqnarray*}
			\big|P(e^{2\pi i t_{r,q}})-P(e^{2 \pi i t_{q,r,\delta}^*})\big|
			\leq q \|P_S\|_{\infty} \Big|e^{2\pi i t_{r,q} }-e^{2 \pi i t_{q,r,\delta}^*}\Big|\\\
			\leq  \frac{\sqrt{p+1}}{2 \pi} \frac{\delta}{p^{1/2+\epsilon}} \tend{q}{\infty}0.
		\end{eqnarray*}
		This combined with the standard triangle inequalities gives
		\begin{eqnarray*}
			&&\Bigg|\Bigg(\frac1{q}\sum_{r=0}^{q-1}\Big|P_S\big(e^{2\pi i t_{q,r}}\big)\Big|^{\alpha}\Bigg)^{\frac1{\alpha}}
			-\Bigg(\frac1{q}\sum_{r=0}^{q-1}\Big|P_S\big(e^{2\pi i t_{q,r,\delta}^*}\big)\Big|^{\alpha}\Bigg)^{\frac1{\alpha}}\Bigg|\\
			&\leq& \Bigg(\frac1{q}\sum_{r=0}^{q-1}\Big|P_S\big(e^{2\pi i t_{q,r}}\big)-P_S\big(e^{2\pi i t_{q,r,\delta}^*}\big)\Big|^{\alpha}\Bigg)^{\frac1{\alpha}}\\
			&\leq& \frac{\sqrt{p+1}}{2 \pi} \frac{ \delta}{p^{1/2+\epsilon}} \tend{q}{\infty}0,
		\end{eqnarray*}
		and the proof of the lemma is complete.
	\end{proof}
	Lemma \ref{Lem3} allow us to construct a new families of nodal points for which \eqref{eqf3} holds. We are going to follow the spirit of this strategy in the proof of our main results.
	\subsection{Proof of Theorems \ref{main1}} Following the spirit of the Kadets $1/4$ theorem for polynomials
	due to Marzo-Seip \cite{Marzo-S}, and
	the very recent refinement of the Marcinkiewicz-Zygmund inequalities, we start by establishing a
	necessary and sufficient conditions for a sequence of the analytic trigonometric polynomials to be $L^\alpha$-flat, $\alpha>0$. \\

	Let $S$ be a fixed Singer set in $\Z/q\Z$ with $q=p^2+p+1$, $p$ prime number and put
	$$P_q(z)=P_S(z).$$
	Define
	$$z_{j,q}=e^{2\pi i\frac{j}{q}},$$
	and for a given $\delta_{q,j}>0,$ $j=0,\cdots,q-1$, we put
	$$z_{r,q,(\delta_{q,j})}^*=e^{2\pi i\big(\frac{j}{q}+\frac{\delta_{q,j}}{q}\big)}.$$
	Let $\delta>0$. We define
	$$ z_{r,2q,\delta}=\left\{
	\begin{array}{ll}
	z_{\frac{r}2,q}, & \hbox{if $r$ is even;} \\
	z_{\frac{r-1}2,q,\delta}^*, & \hbox{if $r$ is odd,}
	\end{array}
	\right.
	$$
	with $\delta_{q,j}=\delta,j=0,\cdots,q-1,$ and $ \rho_q=1-\frac1{2q}$.
	We thus have
	$$\Big\{z_{r,2q,\delta}\Big\}=\Big\{z_{r,q}\Big\}_{r=0}^{q-1} \bigcup \Big\{z_{r,q,\delta}^*\Big\}_{r=0}^{q-1},$$
	and we set
	$$F_{2q-1}(z)=\prod_{r=0}^{2q-1}\Big(1-\rho_q\overline{z_{r,2q,\delta}}z\Big),{\textrm{~~where~~}} \rho_q=\frac{2q-1}{2q}.$$
	We recall that the family of nodal points $\mathcal{Z}=\Big\{{\{z_{j,n}\}}_{j=0}^{n}\Big\}_{n \geq 0}$ is said to be an $L^\alpha$ Marcinkiewicz-Zygmund family if the $L^\alpha$ 
	Marcinkiewicz-Zygmund inequalities holds for the nodal points ${\{z_{j,n}\}}_{j=0}^{n}$, for every $n\geq 0$.\\
	
	We associate to any  family of nodal points $\mathcal{Z}=\Big\{{\{z_{j,n}\}}_{j=0}^{n}\Big\}_{n \geq 0}$ the function $F_n$ defined by
	$$F_n(z)=\prod_{r=0}^{n}\Big(1-\rho_n\overline{z_{r,n}}z\Big),{\textrm{~~where~~}} \rho_n=1-\frac1{n+1}.$$
	
	The family $\mathcal{Z}$ is said to be uniformly separated if there is a positive
	number $c$ such that
	$$\inf_{j \neq k}\big|z_{n,j}-z_{n,k}\big| \geq \frac{c}{n+1},~~~~~~~~~~\forall~n \geq 0.$$
	This notion is related to the notion of Carleson measures and following \cite{Seip-book} the sequence $\mathcal{X}=\{\xi_n\}$ of points in the open unit disk $\D$ is said to satisfy 
	Carleson's condition if,
	\begin{eqnarray}\label{US}
	\gamma=\inf_{k}\prod_{\overset{j=1}{j \neq k}}^{\infty}\Big|\frac{\xi_j-\xi_j}{1-\overline{\xi_k}\xi_j}\Big|>0.
	\end{eqnarray}
	Of course this condition is connected to the well-known Carleson's interpolation theorem \cite{Carleson}. For the proof of Carleson's interpolation theorem, 
	we refer the reader to \cite[p.157]{Duren}, \cite[p.1]{Koosis}, \cite[p.274]{Garnett}.\\
	
	We recall that a finite measure $\mu$ is a Carleson measure if the injection mapping from  $H^\alpha$, $\alpha>0$  to  the space $L^\alpha(\D,\mu)$ is  bounded. Carleson described geometrically 
	these measures in \cite{Carleson-Ann} by proving that the finite measure $\mu$ is a Carleson measure if and only if there exist a 
	constant $C_{\alpha}>0$ such that
	$$\int_{\D} |f(z)|^{\alpha} d\mu \leq C_{\alpha} \big\|f\big\|_{H^\alpha}^{\alpha}~~~~f \in H^\alpha,$$
	for any $\alpha>0$. We refer also to  \cite[p.156]{Duren} for the proof of Carlson's theorem. We further have, by Carleson's interpolation theorem, that the discrete measure $\mu$ given by
	$$\mu=\sum_{n=1}^{+\infty}(1-|z_k|^2)\delta_{z_k},$$
	where $\delta_w$ is the Dirac measure on $w$, is a Carleson measure if the family $\{z_k\}$ is uniformly separated. This result was strengthened in [15] by 
	McDonald and Sundberg \cite{McDonald-S}, who proved that the sequence $\{z_k\}$
	of points in $\D$ generates a discrete Carleson measure $\mu$ if and only if $\{z_k\}$ is a finite union of uniformly separated sequences. For a simple proof, 
	we refer to \cite{Duren-S}. We notice that if the sequence is uniformly separated then the constant $C_{\alpha}$ depend uniquely on $\gamma$.
	In this setting, we have the following lemma
	\begin{lem}\label{CarlsonSomme}The sequences $\mathcal{Z}=\Big\{\big\{\rho_q z_{r,q}\big\}\Big\}_{q \geq 0}$ and  $\mathcal{Z}^*=$ $
		\Big\{\big\{\rho_q z_{r,q,\delta}^*\big\}\Big\}_{q \geq 0}$ are uniformly separated sequences.
	\end{lem}
	\begin{proof}Put
		$$\xi_r^*=\rho_q z_{r,q,\delta}^*= \rho_q e^{it_{r,q}},$$
		where $t_{r,q}=2\pi (\frac{r}{q}+\frac{\delta}{q}),$ $r=0,\cdots,q-1$. Then
		\begin{eqnarray}\label{USI1}
		\Big|\frac{\xi_r^*-\xi_s^*}{1-\overline{\xi_s^*}\xi_r^*}\Big|^2&=&
		\frac{2\rho_q^2\Big(1-\cos\big(t_{r,q}-t_{s,q}\big)\Big)}
		{1-2\rho_q^2\cos\big(t_{r,q}-t_{s,q}\big)+\rho_q^4} \nonumber\\&
		=&\frac{4\rho_q^2\Big(\sin\big(\frac{t_{r,q}-t_{s,q}}2\big)\Big)^2}
		{\big(1-\rho_q^2\big)^2+4\rho_q^2\Big(\sin\big(\frac{t_{r,q}-t_{s,q}}2\big)\Big)^2}\nonumber\\
		&=& \frac{4\rho_q^2\Big(\sin\big(\pi\frac{r-s}{q}\big)\Big)^2}
		{\big(1-\rho_q^2\big)^2+4\rho_q^2\Big(\sin\big(\pi\frac{r-s}{q}\big)\Big)^2}
		\end{eqnarray}
		
		Notice that $\pi.\frac{r-s}{q} \in ]-\pi,\pi[$, and the function $x \mapsto \sin^2(x)$ is an even function. We further have,
		for any $x \in \R$,   $$\sin^2(x-\pi)=\sin^2(x).$$
		Therefore, we can reduce our study to the case of $\pi.\frac{r-s}{q} \in ]0,\pi/2]$,
		and if $\pi.\frac{r-s}{q} \in [\pi/2,\pi[$ we substitute $\pi.\frac{r-s}{q}$ by $\pi.\frac{r-s}{q}-\pi \in [-\pi/2,0[$.\\
		Now, assuming $\pi.\frac{r-s}{q} \in ]0,\pi/2]$, it follows that
		$$ \sin^2\Big(\pi.\frac{r-s}{q}\Big) \geq 4 \frac{(r-s)^2}{q^2},$$
		since, for any $x \in ]0,\pi/2]$, we have
		$\ds \sin(x) \geq \frac{2}{\pi} x.$
		Whence
		\begin{eqnarray}\label{USI2}
		\frac{4\rho_q^2\Big(\sin\big(\pi\frac{r-s}{q}\big)\Big)^2}
		{\big(1-\rho_q^2\big)^2+4\rho_q^2\Big(\sin\big(\pi\frac{r-s}{q}\big)\Big)^2}
		\geq \frac{4\rho_q^2\Big(4 \frac{(r-s)^2}{q^2}\Big)}
		{\big(1-\rho_q^2\big)^2+4\rho_q^2\Big( 4 \frac{(r-s)^2}{q^2}\Big)},
		\end{eqnarray}
		by the monotonicity of the function $\ds \phi(x)=\frac{4 \rho_q^2 x}{\big(1-\rho_q^2\big)^2+4\rho_q^2 x}$.\\ We further have
		$$\big(1-\rho_q^2\big)^2 \leq \frac{8\rho_q^2}{q^2} \leq \frac{16\rho_q^2}{q^2},$$
		since, for any $n \geq 1$, $(n-1) \leq 2 \sqrt{2}(n-1)$.
		This combined with \eqref{USI1} and \eqref{USI2} gives
		\begin{eqnarray}\label{eqzhong}
		\Big|\frac{\xi_r^*-\xi_s^*}{1-\overline{\xi_s^*}\xi_r^*}\Big|^2 \geq\frac{(r-s)^2}{1+(r-s)^2}.
		\end{eqnarray}
		We thus get
		$$\inf_{s}\prod_{\overset{r=0}{r \neq s}}^{q-1}\Big |\frac{\xi_r^*-\xi_s^*}{1-\overline{\xi_s^*}\xi_r^*}\Big|^2
		\geq \prod_{t=1}^{+\infty}(1-\frac1{1+t^2})\setdef \gamma^2>0,$$
		by the convergence of $\ds \sum_{t=1}^{+\infty}\frac1{1+t^2}.$
	\end{proof}
	It follows from Lemma \ref{CarlsonSomme} that the union of the families $\mathcal{Z}$ and $\mathcal{Z}^*$ generates a Carleson measure since the sum of 
	two Carleson measures is a Carleson measure. We further deduce the following\\
	\begin{lem}\label{SULem}The sequence $\mathcal{Z}=\Big\{\big\{\rho_q z_{r,q}\big\}\Big\}_{q \geq 0} \cup
		\Big\{\big\{\rho_q z_{r,q,\delta}^*\big\}\Big\}_{q \geq 0}$ is uniformly separated, and we have
		$$\inf_{\xi \in \mathcal{Z}}\prod_{\overset{\chi \in \mathcal{Z}}{\chi\neq \xi}} \Big |\frac{\chi-\xi}{1-\overline{\xi}\chi}\Big| \geq \gamma^2.\frac{\delta}{\sqrt{1+\delta^2}},$$
		where
		$$\gamma^2=\prod_{t=1}^{+\infty}\Big(1-\frac1{1+t^2}\Big).$$
	\end{lem}
	\begin{proof}Put
		$$\xi_r=\rho_q z_{r,q}, \textrm{~~and~~} \xi_r^*=\rho_q z_{r,q,\delta}^*,$$
		and let $\xi_s \in  \mathcal{Z},$ $s=0,\cdots q-1$. Then, either $\xi_s \in \big\{\rho_q z_{r,q}\big\}$ or $\xi_s \in \big\{\rho_q z_{r,q,\delta}^*\big\}$. Assuming that
		$\xi_s \in \big\{\rho_q z_{r,q}\big\}$, it follows that
		\begin{eqnarray*}
			\prod_{\overset{\chi \in \mathcal{Z}}{\chi\neq \xi}} \Big |\frac{\chi-\xi}{1-\overline{\xi}\chi}\Big|
			&=&\prod_{ r \neq s} \Big |\frac{\xi_r-\xi_s}{1-\overline{\xi_s}\xi_r}\Big|
			\prod_{ r = 0}^{q-1} \Big |\frac{\xi_r^*-\xi_s}{1-\overline{\xi_s}\xi_r^*}\Big|\\
			&=& \prod_{ r \neq s} \Big |\frac{\xi_r-\xi_s}{1-\overline{\xi_s}\xi_r}\Big|
			\prod_{ r \neq s} \Big |\frac{\xi_r^*-\xi_s}{1-\overline{\xi_s}\xi_r^*}\Big|
			\Big|\frac{\xi_s^*-\xi_s}{1-\overline{\xi_s}\xi_s^*}\Big|\\
			&\geq& \gamma^2 \frac{\delta}{\sqrt{1+\delta^2}},
		\end{eqnarray*}
		by the same arguments as in Lemma \ref{CarlsonSomme} (see also \cite{Zhong}) combined with \eqref{eqzhong}. The same conclusion can be drawn for the case
		$\xi_s \in \big\{\rho_q z_{r,q,\delta}^*\big\}$ since the two sets plays symmetric roles. The proof of the lemma is complete.
		
	\end{proof}
	According to Chui-Zhong's theorem \cite{Zhong} the family $\mathcal{X}=$\linebreak$\{\{\xi_{n,j}\}_{j=0}^{n}\}_{n \geq 0}$ of the points on the unit circle is an $L^\alpha$ 
	Marcinkiewicz-Zygmund family if and only if it is uniformly separated and there exist a constant
	$K_\alpha$ such that
	\begin{eqnarray}\label{Ap}
	\Big(\frac1{|I|}\int_I |F_n(e^{i\theta})|^{\alpha} d\theta\Big)^{\frac1{\alpha}}
	\Big(\frac1{|I|}\int_I |F_n(e^{i\theta})|^{-\frac{\alpha}{\alpha-1}} d\theta\Big)^{\frac1{\alpha}} \leq K_\alpha
	\end{eqnarray}
	For every subarc $I$ of the unit circle and every $n \geq 0$.\\
	We notice that the fact that the family is uniformly separated insure that this family generates a Carleson measure, and it is turn out that
	the second condition \eqref{Ap} is well-known as $A_\alpha$ condition in the setting of the BMO spaces (Bounded Mean Oscillation) \cite[p.215]{Garnett}.
	We remind that locally integrable positive function $w$ satisfy $A_\alpha$ condition if
	\begin{eqnarray}\label{Ap2}
	\sup_{I}\Big(\frac1{|I|}\int_I w(x) dx\Big) \Big(\frac1{|I|}\int_I w(x)^{-\frac1{\alpha-1}} dx\Big)^{\alpha-1}<\infty.
	\end{eqnarray}
	
	It turns out that for $p=2$ the condition \eqref{Ap2} is equivalent to the following Helson-Szeg\"{o} condition:
	\begin{Cond1}
		There are real-valued function $u,v \in L^\infty(\T)$ such that
		$$\|v\|_{\infty}<\frac{\pi}2~~~{\textrm{and}}~~~~ w=e^{u+\widetilde{v}}, ~~~~~~~~~~{\textrm{(HS)}}$$
		where $\widetilde{v}$ is the conjugate function of $v$.
	\end{Cond1}
	For the proof of the equivalence of \eqref{Ap2} when $p=2$ and $(HS)$, we refer to \cite[p.246-259]{Garnett}. Therein, the reader can found also the proof of the prediction Helson-Szeg\"{o}'s theorem
	related to (HS) \cite{Helson-S}.\\
	
	Now, according to the equivalence of $\eqref{Ap2}$ when $p=2$ and $(HS)$, Marzo and Seip \cite{Marzo-S} observe that in order to prove that the condition \eqref{Ap} holds it 
	suffices to establish that the following uniform Helson-Szeg\"{o} condition holds.
	\begin{Cond2}
		There exist sequence $u_n$ and $v_n$ of real-valued function in $L^\infty(\T)$ such that
		$$\sup_n\|v_n\|_{\infty}<\frac{\pi}2,~~~ \sup_n\|u_n\|_{\infty}<+\infty~~~~{\textrm{and}}~~~~ |F_n|^2=e^{u_n+\widetilde{v_n}},$$
		where $\widetilde{v_n}$ is the conjugate function of $v_n$.
	\end{Cond2}
	We are going to prove that the uniform Helson-Szeg\"{o} condition holds. Let $\kappa>0$ and $n=2q-1$. We claim first that we have
	$$|F_n(e^{i \theta})|^2=e^{u_{n,\kappa}(\theta)}|F_{n,\kappa}(e^{i \theta})|^2,$$
	where
	$$F_{n,\kappa}(z)=\prod_{r=0}^{n}\Big(1-\rho_{n,\kappa}\overline{z_{r,n+1,\delta}}z\Big){\textrm{~~and~~}} \rho_{n,\kappa}=
	\max\Big\{\frac12,1-\frac{\kappa}{n+1}\Big\}.$$
	Indeed, the Mahler measure of the function $\ds \phi_{n,\kappa}(\theta)\setdef\frac{F_n(e^{i \theta})^2}{F_{n,\kappa}(e^{i \theta})^2}$ verify
	$$M(|\phi_{n,\kappa}|)=\prod_{r=0}^{n}\Big(\frac{M\big(1-\rho_{n}\overline{z_{r,n+1,\delta}}e^{i\theta}\big)}
	{M\big(1-\rho_{n,\kappa}\overline{z_{r,n+1,\delta}}e^{i\theta}\big)}\Big)=1,$$
	by Proposition \ref{basic} combined with the well-know Jensen formula. We further have
	\begin{eqnarray*}
		\frac{1-\rho_{n}\overline{z_{r,n+1,\delta}}e^{i\theta}}
		{1-\rho_{n,\kappa}\overline{z_{r,n+1,\delta}}e^{i\theta}}&=&\big(1-\rho_{n}\overline{z_{r,n+1,\delta}}e^{i\theta}\big)
		\Big(\sum_{l=0}^{+\infty}\rho_{n,\kappa}^l \overline{z_{r,n+1,\delta}}^le^{i l \theta}\Big)\\
		&=&1+\sum_{l=0}^{+\infty}\rho_{n,\kappa}^{l-1}\overline{z_{r,n+1,\delta}}^l \big(\rho_{n,\kappa}-\rho_{n}\big)e^{i l \theta}.
	\end{eqnarray*}
	Therefore $\phi_{n,\kappa}$ is in $H^1$ and $\log|\phi_{n,\kappa}|$ is integrable. Put
	$$G(z)=\exp\Big(\frac1{2\pi}\int_{0}^{2\pi}\frac{e^{i\theta}+z}{e^{i\theta}-z}\log|\phi_{n,\kappa}(\theta)| d\theta\Big).$$
	Then $G$ is an analytic function in the unit disc $\D$ and
	$\big|G\big|=e^{u_{n,\kappa}}$, where $u_{n,\kappa}$ is the Poisson integral of $\log(|\phi_{n,\kappa}|)$, that is,
	$u_{n,\kappa}(re^{i\theta})=P_r*\log(|\phi_{n,\kappa}|)$ where $P_r$ is the Poisson kernel and $*$ is the convolution operator. By Fatou theorem \cite[p.34]{Hoffman},
	$|G|=e^{u_{n,\kappa}}=|\phi_{n,\kappa}|$ almost everywhere on the unit circle $\T$.  We further have
	\begin{eqnarray*}
		u_{n,\kappa}(\theta) &=& 2\Rep\Big(\Log\big(F_n(\theta)\big)-
		\Log\big(F_{n,\kappa}(\theta)\big)\Big)\\
		&=&2 \Rep\Big(\sum_{r=0}^{n}\Big(
		\Log\Big(1-\rho_n\overline{z_{r,n+1,\delta}}e^{i \theta}\Big)-\Log\Big(1-\rho_{n,\kappa}\overline{z_{r,n+1,\delta}}e^{i \theta}\Big)\Big)\Big)\\
		&=& 2\Rep\Big(\sum_{r=0}^{n}\Big(\sum_{l=1}^{+\infty}\frac{\rho_{n,\kappa}^l-\rho_{n}^l}{l}\overline{z_{r,n+1,\delta}}^le^{il\theta}\Big)\Big),
	\end{eqnarray*}
	where $\Log$ is {\it{the principal value of the logarithm}}. Writing
	$$u_{n,\kappa}(\theta) =2\Rep(I+II),$$
	where
	$$I=\sum_{r=0}^{q-1}\sum_{l=1}^{+\infty}\frac{\rho_{n,\kappa}^l-\rho_{n}^l}{l}\overline{z_{r,q}}^le^{il\theta},
	{\textrm {~and~~}}
	II= \sum_{r=0}^{q-1}\sum_{l=1}^{+\infty}\frac{\rho_{n,\kappa}^l-\rho_{n}^l}{l}\overline{z_{r,q}}^l e^{-2i\pi \frac{l\delta}{2q}}e^{il\theta}.$$
	It follows that
	$$I=\sum_{l=1}^{+\infty}\frac{\rho_{n,\kappa}^{lq}-\rho_{n}^{lq}}{l}e^{ilq\theta},
	{\textrm {~and~~}} II= \sum_{l=1}^{+\infty}\frac{\rho_{n,\kappa}^{lq}-\rho_{n}^{lq}}{l} e^{-2il\pi\frac{\delta}{2}}e^{ilq\theta},$$
	since
	$$\sum_{r=0}^{q-1}\overline{z_{r,q}}^l=\left\{
	\begin{array}{ll}
	q, & \hbox{if $l \in q\Z$;} \\
	0, & \hbox{if not.}
	\end{array}
	\right.
	$$
	We can thus write
	\begin{eqnarray*}
		\big|u_{n,\kappa}\big| &\leq& |I|+|II|\\
		&\leq&2\Big(\sum_{l=1}^{+\infty}\frac{\rho_{n,\kappa}^{lq}}{l}+\sum_{l=1}^{+\infty}\frac{\rho_{n}^{lq}}{l}\Big)\\
		&\leq&2 \Big(\log\Big(\frac1{1-\rho_{n,\kappa}^{q}}\Big)+\log\Big(\frac{1}{1-\rho_{n}^{q}}\Big)\Big),\\
		&\leq& \frac{2}{1-e^{-\frac{\kappa}{2}}}+\frac{2}{1-e^{-\frac{1}{2}}}
	\end{eqnarray*}
	since $\log(x) \leq x$ for any $x>0$, and $\log(1-x) \leq -x$ for $0\leq x< 1$
	%$ \rho_{n,\kappa}^{q}=\ds \Big(1-\frac{\kappa}{2q}\Big)^q\sim e^{-\frac{\kappa}2}$
	. We thus conclude that
	$$\sup_{n}\big\|u_{n,\kappa}\big\|_{\infty}<+\infty,$$
	and the proof of the claim is complete.\\

	We move now to construct the functions $v_n$. For that, we start by proving the following lemma.
	\begin{lem}Let $F(z)=(1-r z_0 z)$, with $0<r<1$ and $z_0=e^{i \theta_0}$. Then
		$|F|^2(e^{i \theta})=e^{\widetilde{v}}$, where $\widetilde{v}$ is the conjugate function of the function $v$ given by
		\begin{eqnarray}\label{FLem}
		v(\theta)=P_r*\1_{[0,\theta]}(\theta_0)-\theta-c,
		\end{eqnarray}
		and $c$ is any suitable constant.
	\end{lem}
	\begin{proof} Obviously, $F^2$ is an outer function since the zeros of $F^2$ are out of the disc $\D$. We further have $F^2(0)=1$. Whence $|F|^2=e^{\widetilde{v}}$, 
		where $\widetilde{v}$ is the conjugate function of the function $v$ given by \eqref{FLem}. Indeed, for any $\theta$, we have
		\begin{eqnarray}\label{Log}
		\nonumber \Log\big(F(e^{i\theta})\big)&=&-\sum_{n=1}^{+\infty}\frac{r^n}{n} e^{i n (\theta-\theta_0)} \\
		&=&-\sum_{n=1}^{+\infty}\frac{r^n}{n} \cos\big({ n (\theta-\theta_0)}\big)+i
		\sum_{n=1}^{+\infty}\frac{r^n}{n} \sin\big({ n (\theta-\theta_0)}\big),
		\end{eqnarray}
		where $\Log$ is {\it{the principal value of the logarithm}}. We further have
		\begin{eqnarray*}
			v(\theta)&=&P_r*\1_{[0,\theta]}(\theta_0)-\theta-c=\int_{0}^{\theta}P_r(\theta_0-t) dt-\theta-c\\
			&=& \int_{0}^{\theta}P_r(t-\theta_0) dt -\theta-c,
		\end{eqnarray*}
		Since $P_r$ is an even function. But
		\begin{eqnarray}\label{PLog1}
		\int_{0}^{\theta}P_r(t-\theta_0) dt &=&\int_{0}^{\theta}\sum_{n=-\infty}^{+\infty}r^{|n|}e^{in(t-\theta_0)} dt\nonumber \\
		&=&\sum_{n \neq 0}\frac{r^{|n|}}{in}e^{in(\theta-\theta_0)}-\sum_{n \neq 0}\frac{r^{|n|}}{in}e^{-in\theta_0}+\theta,
		\end{eqnarray}
		by \eqref{Poisson}. Consequently
		$$v(\theta)=\sum_{n \neq 0}\frac{r^{|n|}}{in}e^{in(\theta-\theta_0)}-\sum_{n \neq 0}\frac{r^{|n|}}{in}e^{-in\theta_0}-c,$$
		and
		\begin{eqnarray}\label{conj}
		\widetilde{v}(\theta)&=&-i\sum_{n \neq 0}\frac{n}{|n|}\frac{r^{|n|}}{in}e^{in(\theta-\theta_0)},\nonumber\\
		&=& -2 \sum_{n = 1}^{+\infty}\frac{r^{n}}{n}\cos\big(n(\theta-\theta_0)\big)
		\end{eqnarray}
		Combining \eqref{Log} with \eqref{conj}, we obtain
		$$\Rep\Big(\Log(F^2(e^{i \theta}))\Big)=\widetilde{v}(\theta),$$
		which gives
		$$|F^2|=e^{\widetilde{v}},$$
		since for any analytic function $g$, we have
		$$|e^{g}|=e^{\Rep\big(g\big)},$$
		and the proof of the lemma is complete.
	\end{proof}
	We now apply Lemma \ref{FLem} to write
	$$|F_{n,\kappa}(e^{i \theta})|^2=e^{\widetilde{v_{n,\kappa}}(\theta)},$$
	where
	\begin{eqnarray}\label{HSD1}
	v_{n,\kappa}(\theta)=\sum_{j=0}^{n}\int_{0}^{\theta}P_{\rho_{n,\kappa}}\big(\theta_{n,j}-t\big)dt-(n+1)\theta-c,
	\end{eqnarray}
	$$\theta_{n,j}=\left\{
	\begin{array}{ll}
	2\pi\frac{j}{2q}, & \hbox{if $j$ is even;} \\
	2\pi\big(\frac{j-1}{2q}+\frac{\delta}{q}\big), & \hbox{if $j$ is odd,}
	\end{array}
	\right.
	$$
	and $c$ is any suitable constant. Taking
	$$c=\sum_{j=0}^{n}\int_{-2\pi \gamma_{n,j}}^{0}P_{\rho_{n,\kappa}}\Big(\Big(2\pi\frac{j-1}{2q}\Big)-t\Big)dt,$$
	with
	$$\gamma_{n,j}=\left\{
	\begin{array}{ll}
	0, & \hbox{if $j$ is even;} \\
	2\pi \frac{\delta}{q}, & \hbox{if $j$ is odd.}
	\end{array}
	\right.
	$$
	We can rewrite \eqref{HSD1} as
	\begin{eqnarray}\label{HSD2}
	v_{n,\kappa}(\theta)=\sum_{j=0}^{n}\int_{0}^{\theta-2\pi\gamma_{n,j}}P_{\rho_{n,\kappa}}\big(\theta_{n,j}-t\big)dt-(n+1)\theta,
	\end{eqnarray}
	since, for any odd $j$, we have
	\begin{eqnarray*}
		\int_0^{\theta}P_{\rho_{n,\kappa}}(\theta_{n,j}-t)dt&=& \int_0^{\theta}P_ {\rho_{n,\kappa}}\Big(\Big(2\pi\Big(\frac{j-1}{2q}\Big)+\gamma_{n,j}\Big)-t\Big)dt,
		\\
		&=& \int_{-2\pi\gamma_{n,j}}^{\theta-2\pi\gamma_{n,j}}P_{\rho_{n,\kappa}}\Big(\Big(2\pi\Big(\frac{j-1}{2q}\Big)\Big)-t\Big)dt
	\end{eqnarray*}
	Again writing
	$$v_{n,\kappa}(\theta)=I+II,$$
	where
	$$I=\sum_{r=0}^{q-1} \int_{0}^{\theta} P_{\rho_{n,\kappa}}\big(\frac{2\pi j}{q}-t\big)dt-q\theta {\textrm {~~and~~}}
	II= \sum_{r=0}^{q-1} \int_{0}^{\theta-\frac{2\pi \delta}{q}} P_{\rho_{n,\kappa}}\big(\frac{2\pi j}{q}-t\big)dt-q\theta.
	$$
	We thus need to estimate $|I|$ and $|II|$. But, by the same reasoning as above, it is easy to check that
	$$I= \sum_{l \neq 0}{\rho_{n,\kappa}^{lq}}\frac{1-e^{-ilq\theta}}{il} {\textrm {~and~~}}
	II= \sum_{l \neq 0} {\rho_{n,\kappa}^{lq}}\frac{1-e^{-ilq(\theta-\frac{2\pi \delta}{q})}}{il}-2\pi\delta.
	$$
	Consequently, we get
	$$|I|\leq 4 \sum_{l\geq 1}\frac{\rho_{n,\kappa}^{lq}}{l}=-4\log(1-{\rho_{n,\kappa}^{q}}).$$
	Whence
	$$ |I|\lesssim -4\log\big({1-e^{-\frac{\kappa}2}}\big).$$
	It is still to estimate $|II|$. In the same manner, it can be seen that
	$$|II| \lesssim 2\pi \delta+4 \sum_{l\geq 1}\frac{\rho_{n,\kappa}^{lq}}{l} \leq 2 \pi \delta-4\log\big({1-e^{-\frac{\kappa}2}}\big),$$
	and by choosing $\kappa$ sufficiently large and $\delta<\frac18$, we obtain
	$$\sup_{n}\|v_{n,\kappa}(\theta)\|_{\infty}<\frac{\pi}2.$$
	From this we conclude that the uniform Helson-Szeg\"{o} condition holds for $\alpha=2$.\\
	For the case $1<\alpha \neq 2$. Assuming $\delta<\frac1{8\beta}$ where $ \beta=\max\big\{\alpha,\frac{\alpha}{\alpha-1}\big\}$,
	one may apply the standard argument from the $H^p$ theory combined with the H\"{o}lder inequality and Lemma 2 from \cite{Marzo-S}
	to conclude that the uniform Helson-Szeg\"{o} condition holds.\\
	
	%observe that by Lemma 2 from \cite{Marzo-S}, we can write
	Now, let $1<\alpha<2$ and $0<\delta<\frac1{4\beta}$, with $\beta=\frac{\alpha}{\alpha-1}$.
	%By the same reasoning as in Lemma \ref{CarlsonSomme}, it is a simple matter to check that
	By Lemma \ref{SULem} the family $\mathcal{Z}=\Big\{\big\{\rho_q z_{r,q}\big\}\Big\}_{q \geq 0} \cup
	\Big\{\big\{\rho_q z_{r,q,\delta}^*\big\}\Big\}_{q \geq 0}$ is uniformly separated. We can thus write
	\begin{eqnarray}\label{Key3}
	&&\int \Big|\big|P_q(\theta)\big|^2-1\Big|^\alpha d\theta \nonumber \\
	&& \leq C_{\alpha,\delta}\Big(\frac1{2q}\sum_{r=0}^{q-1} \Big|\big|P_q\big(z_{r,q})|^2-1\Big|^\alpha+
	\frac1{2q}\sum_{r=0}^{q-1} \Big|\big|P_q\big(z_{r,q,\delta}^*)|^2-1\Big|^\alpha\Big),
	\end{eqnarray}
	where
	$$C_{\alpha,\delta}=\Big(\frac{2C_{\gamma}}{\gamma^2. \frac{\delta}{\sqrt{1+\delta^2}}}\Big)^{\alpha}=
	\Big(\frac{2C_{\gamma} \sqrt{1+\delta^2}}{\gamma^2. \delta}\Big)^{\alpha}.$$
	The computation of constant $C_{\alpha,\delta}$ can be found in \cite[p.153]{Duren}. Therein, by appealing to the duality argument, it is shown that for any
	$w=(w_j)\in \ell^{\alpha}$, there exist $g \in H^{\alpha}$ such that for some $f \in H^{\beta}$, with $\big\|f\big\|_{\beta}=1$,
	and $\beta$ is the conjugate of $\alpha$, one can assert
	$$\big\|g\big\|_\alpha \leq \frac{\sqrt{1+\delta^2}}{\gamma^2.\delta}\|w\|_{\alpha}\Big(\int_{\D} |f(z)|^{\beta}d\mu_{\mathcal{Z}}
	+\int_{\D} |f(z)|^{\beta}d\mu_{\mathcal{Z}^*}\Big),$$
	where
	$$\mu_{\mathcal{Z}}=\sum_{q=3}^{+\infty} \Big(\sum_{r=0}^{q-1}
	\big(1-|\rho_q z_{r,q}|\Big) \delta_{z_{r,q}}, \textrm{~~and~~}
	\mu_{\mathcal{Z}^*}=\sum_{q=3}^{+\infty} \Big(\sum_{r=0}^{q-1}
	\big(1-|\rho_q z_{r,q,\delta}^*|\Big) \delta_{z_{r,q}}.$$
	But the measures $\mu_{\mathcal{Z}}$ and $\mu_{\mathcal{Z}^*}$ are a Carleson measures. Therefore,
	$$\int_{\D} |f(z)|^{\beta}d\mu_{\mathcal{Z}}
	+\int_{\D} |f(z)|^{\beta}d\mu_{\mathcal{Z}^*} \leq 2 C_{\gamma}\|f\|_{\beta}=2 C_{\gamma},$$
	where
	$$C_{\gamma}=\frac2{\gamma^4}\big(1-2\log(\gamma)).$$
	An alternative proof can be found in \cite[p.195-202]{Hoffman}. The reader may notice that the proof of Theorem F in \cite{ZhongII}
	can be drawn from the above proof. \\
	Let us stress at this point that we have proved that for any $\alpha>0,$ the sequence of polynomials $(P_q(z))$ is $L^\alpha$-flat if and only if
	$$\frac1{q}\sum_{r=0}^{q-1} \Big|\big|P_q\big(z_{r,q})|^2-1\Big|^\alpha \tend{q}{+\infty} 0,$$
	and
	$$\frac1{q}\sum_{r=0}^{q-1} \Big|\big|P_q\big(z_{r,q,\delta}^*)|^2-1\Big|^\alpha \tend{q}{+\infty} 0.$$
	
	We have also proved Lemma 12. in \cite{Ortega}. This Lemma will be stated and used later.\\\
	
	Now, we are going to prove that $(P_q)$ is $L^\alpha$-flat, $1<\alpha<2$. We start by writing 
	\begin{eqnarray}\label{eqSinger:1}
	&&\Big|\Big(\frac1{q}\sum_{r=0}^{q-1}\Big|\big|P_q\big(z_{r,q})|^2-1\Big|^\alpha\Big)^{\frac1{\alpha}}- 
	\Big(\frac1{q}\sum_{r=0}^{q-1}\Big|\big|P_q\big(z_{r,q,\delta}^*)|^2-1\Big|^\alpha\Big)^{\frac1{\alpha}}\Big| \nonumber\\
	&&\leq \Big(\frac1{q}\sum_{r=0}^{q-1}\Big|\big|P_q\big(z_{r,q})|^{2}-\big|P_q\big(z_{r,q,\delta}^*)|^{2}\Big|^{\alpha}\Big)^{\frac1{\alpha}}
	%\\
	%&&\leq 2^{\frac1{\alpha}} \Big(\frac1{2q}\sum_{r=0}^{2q-1}\Big|\big|P_q\big(z_{r,2q})|^{2}-\big|P_q\big(z_{r,2q,2\delta}^*)|^{2}\Big|^{\alpha}\Big)^{\frac1{\alpha}}
	%\\
	%&& \leq c_{\alpha}\frac{\delta}{2q} \Big(\int \Big|\Big(|P(\theta)|^2\Big)'\Big|^\alpha d\theta\Big)^{\frac1{\alpha}}.
	\end{eqnarray}
	Moreover, by the nice properties of Singer sets, we have
	\begin{align}\label{ineq-main}
	&\Big(\frac1{q}\sum_{r=0}^{q-1}\Big|\big|P_q\big(z_{r,q})|^{2}-\big|P_q\big(z_{r,q,\delta}^*)|^{2}\Big|^{\alpha}\Big)^{\frac1{\alpha}} \tend{q}{+\infty}0.
	% &\leq C'_{\alpha,\delta} \Big(\frac1{q}\sum_{r=0}^{q-1}\Big|\frac{1}{|S|}D_q^{*}\big(z_{r,q}\big)-\frac{1}{|S|}D_q^{*}\big(z_{r,q,\delta}^*)\Big|^{\alpha}\Big)^{\frac1{\alpha}}+\frac{2\pi \delta}{|S|}\\
	%&=\Big(\frac1{q}\sum_{r=0}^{q-1}\Big|\frac{1}{|S|}\sum_{\ell=0}^{q-1}z_{r,q}^{\ell}\big(1-e^{2 \pi i \frac{\ell. \delta}{q}}\big)\Big|^{\alpha}\Big)^{\frac1{\alpha}} 
	%&\lesssim \Big(\frac1{q}\sum_{r=0}^{q-1}\Big|\frac{1}{|S|}D_q^{*}\big(z_{r,q}\big)-\frac{1}{|S|}D_q^{*}\big(z_{r,q,\delta}^*)\Big|^{\alpha}\Big)^{\frac1{\alpha}}+\frac{2\pi \delta}{|S|}\\ 
	%&\leq C'_{\alpha,\delta} \Bigg(\Big(\frac1{q}\sum_{r=0}^{q-1}\Big|\frac{1}{|S|}D_q^{*}\big(z_{r,q}\big)\Big|^\alpha\Big)^{\frac1{\alpha}}+
	%\Big(\frac1{q}\sum_{r=0}^{q-1}\Big|\frac{1}{|S|}D_q^{*}\big(z_{r,q,\delta}^*)\Big|^{\alpha}\Big)^{\frac1{\alpha}}\Bigg)+\frac{2\pi \delta}{|S|}
	\end{align}

	\noindent Indeed,  write
	$$|P_q(z)|^2-1=\frac{1}{|S|}\sum_{l=1}^{q-1}c_{l}z^l+\frac{1}{|S|}\sum_{l=1}^{q-1}c_{-l}z^{-l},$$
	where $(c_l)$ are the correlation of the sequence $\Big(\1_S(j)\Big)_{j=0}^{q-1}$ given by 
	$$c_l=\sum_{\{s,t~:~s-t=l\}}\1_S(s)\1_S(t),$$
	for $l=1,\cdots, q-1$ and  put
	$$Q_q(z)=\frac{1}{|S|}\sum_{l=1}^{q-1}c_{l}z^l.$$
	Observe that formally we have $$c_{l}=\sum_{j=0}^{q-l-1}\mathbbm{1}_{S}(j) \mathbbm{1}_{S}(j+l)=
	\Big|\{(j,k) \in S \times S\;:\; j-k=l\}\Big|$$ and 
	$$c_{-l}=\sum_{j=l}^{q-1}\mathbbm{1}_{S}(j) \mathbbm{1}_{S}(j-l)
	=\Big|\{(j,k)\in S \times S\;:\; j-k=-l\}\Big|.$$ We further have, by the proof of Lemma \ref{Lem1}, for any $r=1,\cdots,q-1$,
	\begin{align}
	\sum_{l=1}^{q-1}c_{l}z_{r,q}^l+\sum_{l=1}^{q-1}c_{-l}z_{r,q}^{-l}&=
	\Rep\Big(\sum_{l=1}^{q-1}c_{l}z_{r,q}^l+\sum_{l=1}^{q-1}c_{-l}z_{r,q}^{-l}\Big)\\
	&=	\sum_{l=1}^{q-1}(c_{l}+c_{-l})z_{r,q}^l\\
	&=\sum_{l=1}^{q-1}z_{r,q}^l \label{Simple-I}\\
	&=-1
	\end{align}
	The equality \eqref{Simple-I} follow from the fact that mod $q$, $\frac{c_{l}+c_{-l}}{2}=1,$ since 
	$c_{-l}=c_{q-l}=c_l$ and by Singer theorem $(S-S)\setminus\{0\}$ cover mod $q$ the set $[1,q-1]$ as a subset of  $[-(q-1),(q-1)]\setminus\{0\}.$ So, 
	
	$$\sum_{l=1}^{q-1}(c_{l}+c_{-l})z_{r,q}^l=\frac{1}{2}\Big(\sum_{l=1}^{q-1}z_{r,q}^l+
	\sum_{l=1}^{q-1}z_{r,q}^l\Big)=-1.$$
	
	\noindent At this point, we need the following lemma from \cite[Lemma.12]{Ortega}. 
	\begin{lem}\label{MZ-delta}Let $\mathcal{Z}=\{\{z_{n,j}\}_{j=0,\cdots,m_n}, n \geq 0\}$ is a Marcinkiewicz-Zygmund triangular family then there is an $\epsilon > 0$ (depending only on the constants of Marcinkiewicz-Zygmund inequalities for $\Z$) such that for any perturbation ${\mathcal{Z}}^{*}$ of the original family the property  $|z_{n,j}-z_{n,j}^*| \leq \frac{\epsilon}{n}$ is still a Marcinkiewicz-Zygmund triangular family.
	\end{lem}
	We recall that the triangular family is said to be a Marcinkiewicz-Zygmund triangular family if it satisfy the  Marcinkiewicz-Zygmund inequalities  \eqref{MZ}.\\
	
	An alternative proof of Lemma \ref{MZ-delta} can be obtained by observing that $\mu_{\widetilde{\mathcal{Z}^*}}$ is a Carleson measure, and this later argument can be seen by applying our previous arguments combined with Bernstein-Zygmund inequalities (Theorem \ref{Bernstein}). Therefore, by applying Lemma \ref{MZ-delta}, we get   
	
	%Therefore, by the main value theorem and Bernstein-Zygmund inequalities (Theorem \ref{Bernstein}), 
	%there  are points $\widetilde{z}_{r,q,\delta}$ in between $z_{r,q}$ and
	%$z_{r,q,\delta}^*$ such that (\footnote{See, for instance, \cite[p. 8]{Ortega}.}) 

	\begin{align}\label{BZMZ}
	\Big(\frac1{q}\sum_{r=0}^{q-1}\Big|\frac{1}{|S|}\sum_{l=1}^{q-1}(c_{l}+c_{-l})z_{r,q,\delta}^{*l}\Big|^{\alpha}\Big)^{\frac1{\alpha}}
	&\leq A_{\alpha}^{-1}\Big(\int \Big|\frac{1}{|S|}\sum_{l=1}^{q-1}(c_{l}+c_{-l})z^l\Big|^{\alpha} dz\nonumber \Big)^{\frac{1}{\alpha}}  \\   
	&\leq \frac{B_{\alpha}}{A_{\alpha}} \Big(\frac1{q}\sum_{r=0}^{q-1}\Big|\frac{1}{|S|}\sum_{l=1}^{q-1}(c_{l}+c_{-l})z_{r,q}^{l}\Big|^{\alpha}\Big)^{\frac1{\alpha}}
	\end{align}
	Now, notice that we have  
	\begin{align}
	\frac{1}{|S|}\sum_{l=1}^{q-1}(c_{l}+c_{-l})&=\frac{q-1}{|S|}\\
	&=|P_q(1)|^2-1=|S|-1=p
	\end{align}
	Therefore 
	\begin{align}
	\nonumber&\frac1{q}\sum_{r=0}^{q-1}\Big|\frac{1}{|S|}\sum_{l=1}^{q-1}(c_{l}+c_{-l})z_{r,q}^{l}\Big|^{\alpha}\\
	&=\frac{p^\alpha}{p^2+p+1}+\frac{q-1}{q}\frac{1}{p+1} \tend{p}{+\infty}0,
	\end{align}
	since $\alpha<2.$ Whence 
	\begin{align}\label{BZMZ2}
	\Big(\frac1{q}\sum_{r=0}^{q-1}\Big|\frac{1}{|S|}\sum_{l=1}^{q-1}(c_{l}+c_{-l})z_{r,q,\delta}^{*l}\Big|^{\alpha}\Big)^{\frac1{\alpha}}\tend{q}{+\infty}0.
	\end{align}
	We thus deduce 
	\begin{align}\label{BZMZ3}
	\Big(\frac1{q}\sum_{r=0}^{q-1}\Big|\frac{1}{|S|}\Rep\Big(\sum_{l=1}^{q-1}(c_{l}+c_{-l})z_{r,q,\delta}^{*l}\Big)\Big|^{\alpha}\Big)^{\frac1{\alpha}}\tend{q}{+\infty}0,
	\end{align}
	by the classical inequality $|z| \geq |\Rep(z)|.$  Summarizing, we have proved the following
	\begin{align}\label{BZMZfinal}
	\Big(\frac1{q}\sum_{r=0}^{q-1}\Big|Q_q(z_{r,q,\delta})+
	Q_q(z_{r,q,\delta}^{-1})\Big|^{\alpha}\Big)^{\frac1{\alpha}}\tend{q}{+\infty}0.
	\end{align}
	From this, we obtain, by letting $q \longrightarrow +\infty$, that, for any $1<\alpha<2$.  
	\begin{align}
	\frac1{q}\sum_{r=0}^{q-1}\Big|\big|P_q\big(z_{r,q,\delta}^*)|^2-1\Big|^\alpha
	\tend{q}{+\infty}0
	\end{align}
	and
	\begin{align}
	\Big(\frac1{q}\sum_{r=0}^{q-1}\Big|\big|P_q\big(z_{r,q})|^2-1\Big|^\alpha\Big)^{\frac1{\alpha}}
	\tend{q}{+\infty}0.
	\end{align}
	%Indeed, by the triangle inequality, we have 
	
	%\begin{eqnarray*}
	%&&\Big|\Big(\frac1{q}\sum_{r=0}^{q-1}\Big|\big|P_q\big(z_{r,q})|^2-1\Big|^\alpha\Big)^{\frac1{\alpha}}- 
	%\Big(\frac1{q}\sum_{r=0}^{q-1}\Big|\big|P_q\big(z_{r,q,\delta}^*)|^2-1\Big|^\alpha\Big)^{\frac1{\alpha}}\Big| \nonumber\\
	%&&\leq \Big(\frac1{q}\sum_{r=0}^{q-1}\Big|\big|P_q\big(z_{r,q})|^{2}-\big|P_q\big(z_{r,q,\delta}^*)|^{2}\Big|^{\alpha}\Big)^{\frac1{\alpha}}\tend{q}{+\infty}0.
	%\\
	%&&\leq 2^{\frac1{\alpha}} \Big(\frac1{2q}\sum_{r=0}^{2q-1}\Big|\big|P_q\big(z_{r,2q})|^{2}-\big|P_q\big(z_{r,2q,2\delta}^*)|^{2}\Big|^{\alpha}\Big)^{\frac1{\alpha}}
	%\\
	%&& \leq c_{\alpha}\frac{\delta}{2q} \Big(\int \Big|\Big(|P(\theta)|^2\Big)'\Big|^\alpha d\theta\Big)^{\frac1{\alpha}}.
	%\end{eqnarray*}
	\noindent We thus conclude, by virtue of \eqref{Key3}, that for any $1<\alpha<2$,%of  choose $p$ large enough such that
	
	%\begin{eqnarray*}
	%\Big|\Big(\frac1{q}\sum_{r=0}^{q-1}\Big|\big|P_q\big(z_{r,q})|^2-1\Big|^\alpha\Big)^{\frac1{\alpha}}- 
	%	\Big(\frac1{q}\sum_{r=0}^{q-1}\Big|\big|P_q\big(z_{r,q,\delta}^*)|^2-1\Big|^\alpha\Big)^{\frac1{\alpha}}\Big| 
	%	\leq (2\pi \delta)^\frac{2}{\alpha}
	%\end{eqnarray*}
	%Combined these inequalities with Lemma \ref{Lem2}, we can rewrite \eqref{Key3} as
	%\begin{eqnarray*}
	%\int \Big|\big|P_q\big|^2-1\Big|^\alpha d\theta \leq C_{\alpha,\delta} \Big(\frac12 I_{\alpha,p}+\frac12
	%\Big(I_{\alpha,p}^\frac1{\alpha}+\Big(2\pi \delta)^\frac{2}{\alpha} \Big)^\alpha\Big),
	%\end{eqnarray*}
	%where
	%\begin{eqnarray*}
	%I_{\alpha,p}&=&\frac{p^\alpha}{q}+\Big(\frac{q-1}{q}\Big)\Big(\frac{p}{p+1}-1\Big), 
	%\textrm{~~and}\\
	%J_{\alpha,p}&=&\frac{(p+1)^\alpha}{q}+\Big(\frac{q-1}{q}\Big)\Big(\frac{p}{p+1}\Big)
	%\end{eqnarray*}
	
	%Letting $q \longrightarrow +\infty$, we obtain
	%$$\lim_{q \longrightarrow +\infty}\int \Big|\big|P_q\big|^2-1\Big|^\alpha=0 
	%$$
	%d\theta
	%\leq C_{\alpha,\delta} ((2\pi \delta)^\frac{2}{\alpha})^\alpha=\Big(\frac{2C_{\gamma}}{\gamma^2}\Big)^\alpha\big(\sqrt{1+\delta^2}\big)^\alpha \big(2. \pi\big)^2\delta^{2-\alpha},$$
	%and by letting $\delta \longrightarrow 0$, we conclude that
	$$\lim_{q \longrightarrow +\infty}\int \Big|\big|P_q\big|^2-1\Big|^{\alpha} d\theta=0,$$
	Now, by Proposition \ref{basic}, we have, for any $0<\alpha \leq 1$, 
	$$\lim_{q \longrightarrow +\infty}\int \Big|\big|P_q\big|^2-1\Big|^{\alpha} d\theta=0.$$
	%Since
	%$$ \int \Big|\big|P_q\big|^2-1\Big| d\theta \leq 2,~~~~\textrm{and~~}~\alpha>1.$$
	Hence the sequence of polynomials $(P_q)_{q \in \N}$ is $L^\alpha$-flat, for $0<\alpha<2$.\\
	
	\noindent{}Therefore, by appealing to Proposition \ref{mainP}, we deduce that the sequence of polynomials $(P_q(z))$ is almost everywhere flat over some subsequence, and
	the proof of Theorem \ref{main1} is complete.
	\hfill \fbox
	
	\section{Generalized Riesz products}\label{Rieszp}
	The classical notion of Riesz products is based on the notion of dissociation, which can be defined as follows.\\
	
	Consider the polynomial $P(z)=1+z$. Then, we have
	$$P(z)^2=1+z+z+z^2.$$
	For any integer $N \geq 2$ ,  we can write
	$$P(z)P(z^N)=1+z+z^N+z^{N+1}.$$
	In the first case we group terms with the same power of $z$, while in the second case
	all the powers of $z$ in the formal expansion are distinct. In the second case we say
	that the polynomials $P(z)$ and $P(z^N)$ are dissociated. More generally, we have
	\begin{lem}[\cite{Abd-Nad2}]
		If $P(z) = \ds \sum_{j=-m}^m a_jz^j, Q(z) = \ds \sum_{j=-n}^{n}b_jz^j$, $m \leq n$, are two trigonometric  polynomials then
		for some $N$, $P(z)$ and $Q(z^N)$ are dissociated.
	\end{lem}
	It is well know that if the sequence of polynomials $(|P_j|^2)$ is dissociated (each finite product has dissociation property) 
	with constant term equal to 1. Then, the sequence of probability measures  $\ds \Big(\prod_{j=1}^{N}|P_j|^2 dz\Big)$
	converge to some probability measure called a Riesz product and denoted by
	$\ds \prod_{j=1}^{+\infty}|P_j|^2.$\\
	
	More generally, we have the following definition.
	
	\begin{Def}\label{def-1}
		Let $P_1, P_2, \cdots,$ be a sequence of trigonometric polynomials such that
		\begin{enumerate}[(i)]
			\item for any finite sequence $i_1< i_2 < \cdots < i_k$ of natural numbers
			$$\int_{S^1}\Bigl| (P_{i_1}P_{i_2}\cdots P_{i_k})(z)\Bigr|^2dz = 1,$$
			where $S^1$ denotes the circle group and $dz$ the normalized Lebesgue measure on $S^1$,
			\item for any infinite sequence $i_1 < i_2 < \cdots $ of natural numbers the weak limit of the measures
			$\mid (P_{i_1}P_{i_2}\cdots P_{i_k})(z)\mid^2dz, k=1,2,\cdots $ as $k\rightarrow \infty$ exists.
		\end{enumerate}
		Then the measure $\mu$ given by the weak limit of $\mid (P_1P_2\cdots P_k)(z)\mid^2dz $ as $k\rightarrow \infty$
		is called generalized Riesz product of the polynomials $\mid P_1\mid^2,
		\mid P_2\mid^2,\cdots$ and denoted by
		$$\displaystyle  \mu =\prod_{j=1}^\infty \bigl| P_j\bigr|^2.  \eqno (1.1)$$
		
	\end{Def}
	
	It is proved in \cite{Abd-Nad} that if $\ds \mu=\prod_{n=1}^{+\infty}|P_n|^2$, then the absolutely continuous part
	$\frac{d\mu}{dz}$ verify
	
	$$\Bigg\|\prod_{n=1}^{N}|P_n|-\sqrt{\frac{d\mu}{dz}}\Bigg\|_1 \tend{N}{+\infty} 0.$$
	
	Furthermore, the Mahler measure of $\mu$ \footnote{ The Mahler measure of the finite measure $\mu$ on the circle is given by
		$$M\Big(\frac{d\mu}{dz}\Big)=\inf_{P}\|P-1\|_{L^2(\mu)},$$
		where $P$ ranges over all analytic trigonometric polynomials with zero constant term.}
	satisfy
	\begin{eqnarray}\label{ANproduct}
	M\Big(\frac{d\mu}{dz}\Big)=\prod_{n=0}^{+\infty}M(P_n^2).
	\end{eqnarray}
	
	We further have the following notion of generalized Riesz products from dynamical origin \cite{Abd-Nad}.
	\begin{Def}\label{def1}
		A generalized Riesz product $\ds \mu = \prod_{j=1}^\infty\big| Q_j(z)\big|^2$,
		where $\ds Q_j(z) = \sum_{i=0}^{n_j} b_{i,j}z^{r_{i,j}}, b_{i,j} \neq 0, \sum_{i=0}^{n_j}\big| b_{i,j}\big|^2 =1$, $\ds \prod_{j=1}^\infty\big| b_{n_j, j}\big| =0$, is said to be of dynamical origin if
		with $$h_0 = 1, h_1 = r_{n_1,1} +h_0, \cdots , h_j = r_{n_j,j} +h_{j-1}, j \geq 2$$
		it is true that for  $j=1,2,\cdots$,
		$$r_{1,j} \geq h_{j-1}, ~~~r_{i+1,j} - r_{i,j} \geq h_{j-1}.$$
		If, in addition, the coefficients $b_{i,j}$ are all positive, then we say that $\mu$ is of purely dynamical origin.\\
	\end{Def}
	The following is proved in \cite{Abd-Nad} .
	\begin{lem}\label{lem1}
		Given a sequence $\ds P_n = \sum_{j=0}^{m_n} a_{j, n}z^j, , n=1,2,\cdots$ of analytic trigonometric polynomials in $L^2(S^1,dz)$ with 
		non-zero constant terms and $L^2(S^1,dz)$ norm 1, $\ds \prod_{n=1}^\infty \big| a_{m_n, n}\big| =0$. Then there exist a sequence of positive integers $N_1, N_2, \cdots$ such that
		$$\prod_{j=1}^\infty\Big| P_j(z^{N_j})\Big|^2$$
		is a generalized Riesz product of dynamical origin.
	\end{lem}
	
	Applying carefully the previous lemma, the following theorem is proved in \cite{Abd-Nad}.
	
	\begin{Th}\label{th7}   Let $P_j, j =1,2,\cdots$ be a sequence of non-constant polynomials
		of $L^2(S^1,dz)$ norm 1 such that $\ds \lim_{j\rightarrow \infty}\big| P_j(z)\big| =1 $ a.e. $(dz)$ then there exists a subsequence $P_{j_k}, k=1,2,\cdots$ and natural numbers $l_1 < l_2 < \cdots$ such that the polynomials $P_{j_k}(z^{l_k}), k=1,2,\cdots $ are dissociated and
		the infinite product $\ds \prod_{k=1}^\infty \big|P_{j_k}(z^{l_k})\big|^2$ has finite nonzero value a.e $(dz)$.
	\end{Th}
	
	\section{Connection to ergodic theory and rank one transformations}\label{ergodic}
	Using the cut and stack
	procedure described in \cite{Friedman1}, \cite{Friedman2},
	one can construct inductively a family of measure-preserving
	transformations, called rank one transformations or rank one maps, as follows.
	\vskip 0.1cm Let $B_0$ be the unit interval equipped with
	Lebesgue measure. At stage one we divide $B_0$ into $m_0$ equal
	parts, add spacers and form a stack of height $h_{1}$ in the usual
	fashion. At the $k^{th}$ stage we divide the stack obtained at the
	$(k-1)^{th}$ stage into $m_{k-1}$ equal columns, add spacers and
	obtain a new stack of height $h_{k}$. If during the $k^{th}$ stage
	of our construction  the number of spacers put above the $j^{th}$
	column of the $(k-1)^{th}$ stack is $a^{(k-1)}_{j}$, $ 0 \leq
	a^{(k-1)}_{j} < \infty$,  $1\leq j \leq m_{k}$, then we have
	$$h_{k} = m_{k-1}h_{k-1} +  \sum_{j=1}^{m_{k-1}}a_{j}^{(k-1)}.$$
	
	%\begin{figure}[hbp]
	%	\begin{center}
	%		\scalebox{0.5}{\input{rankone.ps_tex} }
	%	\end{center}
	%\end{figure}
	
	\noindent{}Proceeding in this way, we get a rank one map
	$T$ on a certain measure space $(X,{\mathcal B},\mid.\mid)$ which may
	be finite or
	$\sigma-$finite depending on the number of spacers added. \\
	\noindent{} The construction of a rank one map thus
	needs two parameters, $(m_k)_{k=0}^\infty$ (cutting parameter), and $((a_j^{(k)})_{j=1}^{m_k})_{k=0}^\infty$
	(spacers parameter). Put
	
	$$T \stackrel {def}= T_{(m_k, (a_j^{(k)})_{j=1}^{m_k})_{k=0}^\infty}$$
	
	\noindent In \cite{Nadkarni1} and \cite{Klemes2} it is proved that
	the spectral type of this map is given (up to possibly some discrete measure) by
	
	\begin{eqnarray}\label{eqn:type1}
	d\mu  ={\rm{W}}^{*}\lim \prod_{k=1}^n\big| P_k\big|^2dz,
	\end{eqnarray}
	\noindent{}where
	\begin{eqnarray*}
		&&P_k(z)=\frac 1{\sqrt{m_k}}\left(1+
		\sum_{j=1}^{m_k-1}z^{-(jh_k+\sum_{i=1}^ja_i^{(k)})}\right),\nonumber  \\
		\nonumber
	\end{eqnarray*}
	\noindent{}$\rm{W}^{*} \lim$ denotes weak star limit in the space of
	bounded Borel measures on ${\T}$.
	
	\noindent{}As mentioned by Nadkarni in \cite{Nad}, the infinite product
	$$
	\prod_{l=1}^{+\infty}\big|P_{j_l}\big(z)|^2$$
	\noindent{}taken over a subsequence $j_1<j_2<j_3<\cdots,$ also represents the maximal spectral type (up to discrete measure) of some rank one maps. 
	In case $j_l \neq l$ for infinitely many $l$, the maps acts on an infinite measure space.\\
	
	The spectrum of any rank one map is simple and  using a random procedure, D. S. Ornstein
	produced a family of mixing
	rank one maps \cite{Ornstein}. It follows that Ornstein's class of
	maps may possibly contain a candidate for Banach's
	problem.  Unfortunately, in 1993, J.
	Bourgain proved that almost surely Ornstein's
	maps have singular spectrum \cite{Bourgain}. Subsequently, using the
	same methods, I. Klemes \cite{Klemes1}
	showed that the  subclass of staircase maps has singular maximal spectral type. In particular, this subclass contains the
	mixing staircase maps of Adams-Smorodinsky \cite{Adams}. Using a refinement of Peyri\`ere criterion \cite{Peyriere}, I. Klemes \& K. Reinhold proved that the rank one maps have a singular
	spectrum if the inverse of the cutting parameter is not in $\ell^2$ (that is, $\sum_{k=1}^{+\infty} \frac1{m_k^2}=+\infty$, where $(m_k) \subset \Big\{2,3,4,\cdots\Big\}$  
	is the cutting parameter) \cite{Klemes2}. This class contains the mixing staircase maps of Adams \& \linebreak
	Friedman \cite{Adams2}. In 1996, H. Dooley and S. Eigen adapted the Brown-Moran methods \cite[pp.203-209]{GrahamMc} and proved that the spectrum of a subclass of 
	Ornstein maps is almost surely singular \cite{DooleyE}.\\
	
	Later, el Abdalaoui-Parreau and Prikhod'ko extended Bourgain theorem \cite{Bourgain} by proving that for any family of probability measures in 
	Ornstein type constructions, the corresponding maps have
	almost surely a singular spectrum \cite{elabdalihp}. They obtained the same result for Rudolph's construction \cite{Rudolph}.
	In 2007, el Abdalaoui showed that the spectrum of the rank one map is singular
	provided that the spacers  $(a_j)_{j=1}^{m_k} \subset \N$,
	are lacunary for all $k$ \cite{elabdaletds}. The author used the Salem-Zygmund central limit theorem methods. 
	As a consequence, the author presented a simple proof of Bourgain's theorem \cite{Bourgain}.\\
	
	Recently, by appealing to a martingale approximation technique, C. Aistleitner and M. Hofer \cite{Aistleitner-H} proved a counterpart of the result in \cite{elabdaletds}. 
	Precisely, they proved that the spectrum of the rank one maps is singular provided that the cutting parameter  $(m_k) \in \N^*$ and the spacers  $(a_j)_{j=1}^{m_k} \subset \N$ satisfy
	\begin{enumerate}[(i)]
		\item $\ds \frac{\log(m_{k_n})}{h_{k_n}}$ converge to $0$;
		\item the proportion of equal terms in the spacers is at least $c.m_{k_n}$ for some fixed constant $c$ and some subsequence $(k_n)$.\\
	\end{enumerate}
	
	We further recall that I. Klemes \& K. Reinhold in \cite{Klemes2} conjectured
	that all rank one maps have singular spectrum, and in the same spirit,
	C. Aistleitner and M. Hofer wrote in the end of their paper 
	``several authors believe that all rank one transformations could have singular maximal spectral type.''. It seems 
	that this conjecture was formulated since Baxter result \cite{Baxter}, \cite{Thouvenot}. We remind that 
	the cutting and stacking rank one construction may goes back to Ornstein's paper in 1960 \cite{Orn}. Indeed, 
	therein, Ornstein constructed a non-singular map for which there is non $\sigma$-finite measure equivalent to Lebesgue measure. Of course, this example is connected to the example of
	non-singular map with simple Lebesgue component constructed by el Abdalaoui and Nadkarni \cite{Abd-Nad1}. Notice that in \cite{Ornstein}, the rank one maps are called transformations of class one.\\
	
	It follows from Bourgain's observation ((eq 2.15)\cite{Bourgain}) that if the spectral type of any rank one map acting on an infinite measure space is singular then the spectral type of any rank one is singular.\\
	
	Unfortunately, by our main result, this strategy fails. Therefore, the new approaches are needed to tackle this conjecture.\\

	\section{Proof of Theorems \ref{main2}}\label{main2-P}
	Theorem \ref{main2} follows from Theorem \ref{main1} combined with Proposition \ref{mainP}.
	Finally, by \eqref{ANproduct}, we deduce that the spectral type $\sigma$ of the rank one map constructed in Theorem \ref{main2} verify
	$$M\Big(\frac{d\sigma}{dz}\Big)=\prod_{j=0}^{+\infty}M(P_j^2)>0.$$
	Whence
	$$M(P_j) \tend{j}{+\infty}1.$$
	This finishes the proof.  For more details on the construction of rank one map in Theorem \ref{main2}, we refer the reader to \cite{Abd-Nad} and \cite{Abd-Nad1}.
	\hfill \fbox
	
	\begin{rem}
		Obviously,
		%we use the fact that $d(\omega_{n,1}+\omega_{n,2})=d\omega_{n,1}+d\omega_{n,2}$ and
		as in the proof given by Zygmund in \cite[p.29, Chap X]{Zygmund}, 
		%in our proof of Theorem \ref{key}, 
		we take advantage
		of the following classical identity \cite[p.35, Chap II]{Zygmund}
		\[
		\frac1{d}\sum_{j=0}^{d-1}e^{\frac{2\pi i j k}{d} }=\begin{cases}
		0 &{\rm {~if~}} d \nmid k \\
		1 & {\rm {~if~}} d \mid k,
		\end{cases}
		\]
		for any $d,k \geq 1$.\\
		\newline
		The reader may notice that there is some analogies between our proof and the Fast Fourier Transform algorithm (FFT). We refer the reader to \cite{JJ} for more details on the FFT.
	\end{rem}
	Applying Carleson interpolation theory, one can prove that for any $p>0$, there is a constant $C_p>0$ such that, for any
	polynomial $P$ of degree less than $n$,
	$$\frac{C_p^{-1}}{4n}\sum_{j=0}^{4n-1}\big|P(e^{2 \pi i \frac{j}{2n}})\big|^p \leq \big\|P\big\|_p^p.$$
	An alternative proof can be found in \cite{Peller}. Besides this, Marcinkiewicz and Zygmund proved \cite{MZ} that for any $p \geq 1$ and for any polynomial $P$ of degree less or equal than $n$, we have
	$$\Bigg(\frac{1}{2n} \sum_{j=0}^{2n-1}\big|P(e^{2 \pi i \frac{j}{2n}})\big|^p \Bigg)^{\frac{1}{p}}\leq \big(p\pi+1\big)^{\frac{1}{p}}  \big\|P\big\|_p.$$
	To the best of this author's knowledge, the explicit constant for the case $p=0$ seems not to be known. Nevertheless, in the case of the classical Riesz product, if we consider the polynomial
	$$P(\theta)=1+\alpha \cos(n \theta),$$
	where $\alpha$ is non-negative number less than $1$.
	Then
	\begin{eqnarray}\label{eqf7}
	M_{d\omega_{4n+1}}(P)&=&\exp\Bigg(\frac{1}{4n}\sum_{j=0}^{4n-1}
	\log\Big(\big|P(e^{2 \pi i \frac{j}{2n}})\big|\Big)\Bigg)\\
	&\leq& M_{dz}(P). \nonumber
	\end{eqnarray}
	\noindent{}This can be proved as follows.\\
	
	\noindent{}Following \cite{ Khrushchev}, we put
	$$P(\theta)=|Q(e^{i\theta})|^2,$$
	where
	$$Q(z)=\frac{1+a z^n}{1+a^2},~~~~~~~~{\textrm{with}}~~a=\frac{\alpha}{1+\sqrt{1-\alpha^2}}.$$
	It is easily seen that $Q(z)$ does not vanish on the disc $D \setdef \big\{|z| \leq 1\big\}$. We thus get that the
	function $\log\big(\big|Q(e^{i\theta})\big|^2\big)$ is harmonic. Applying the mean property, we obtain
	\begin{eqnarray}\label{eqf6}
	\log\big(\big|Q(0)\big|^2\big)=\frac1{2\pi}\int_{0}^{2\pi}\log(\big|Q(e^{i\theta})\big|^2) d\theta.
	\end{eqnarray}
	Rewriting \eqref{eqf6}, we see that
	$$M(P)=\frac{1}{1+a^2}.$$
	Now, any easy computation shows that
	$$P\big(e^{2 \pi i \frac{j}{2n}}\big)=1+\alpha (-1)^j,$$
	for $j=0,\cdots,2n-1.$\\
	Whence
	$$M(P(z)d\omega_{4n+1}(z))=\sqrt{1-\alpha^2}.$$
	Obviously
	$$\sqrt{1-\alpha^2} \leq \frac1{1+a^2}.$$
	We conclude that \eqref{eqf7} holds.\\
	
	This leads us to ask.\\
	\begin{que}~
		
		\begin{enumerate}
			\item  Can one prove or disapprove that  $C_p^{\frac1{p}}$ converge to $1$ as $p \longrightarrow 0$.
			\item Let $S_p$ be a family of Singer sets, $p$ is a prime number and consider the sequence of polynomials
			$$P_q(z)=\frac1{\sqrt{|S_p|}}\sum_{s \in S_p}z^s,~~~~~|z|=1.$$
			Can one prove or disapprove that the sequence of the Mahler measure of $P_q$ converge to 1.
			\item Let $(X,\B,\mu,T)$ be an ergodic dynamical system where $\mu$ is a finite measure. Can one prove or disapprove that there exist a Borel set
			$A$ with $\mu(A)>0$ such that for $\mu$-almost all $x\in X$,
			$$\int \Big|\frac1{\sqrt{N \mu(A)}}\sum_{j=0}^{N-1}\1_A(T^jx)z^{j}\Big|dz \tend{N}{+\infty}1.$$
			\item In the same setting as in the previous question, can one prove or disapprove that for any measurable $f$ with values $\pm 1$, for $\mu$-almost all $x\in X$, we have
			
			$$\limsup_{N \longrightarrow +\infty} \int \Big|\frac1{\sqrt{N}}\sum_{j=0}^{N-1}f(T^jx)z^{j}\Big|dz <1.$$
			As mentioned in introduction, this problem can be linked to the annealed and quenched business.
		\end{enumerate}
	\end{que}

	\begin{thank}
		The author wishes to express his thanks to Fran\c cois Parreau
		who introduced him to the Riesz products business and to Jean-Paul Thouvenot who introduced him to the world of rank one maps.
		He would like further to express his thanks to Mahendra Nadkarni for many stimulating e-conversations and
		to Jean-Fran\c cois M\'ela for his stimulating questions on the rank one maps.
		The author is deeply indebted to Jean-Paul Thouvenot and Mahendra Nadkarni for their suggestions and remarks
		on the previous versions. He would like to thank William Veech for his e-comments supporting this work \footnote{William A. Veech passed away unexpectedly on Tuesday, August 30, 2016 in Houston, Texas. The last email that I received from him was on August 18, 2016. So, that was a shock to me! }.
		The author Would like also to express his thanks 
		to Mahendra Nadkarni, Mahesh Nerurka, S.G. Dani and Anish Ghosh, the organizer of the international conference of algebra and analysis at Pune university, 
		TIFR and the CBS of Mumbai university where the paper was revised, for the invitation and hospitality. 
	\end{thank}
	
\end{document}